\documentclass[10pt]{amsart}
\usepackage{egastyle}

\usepackage{amssymb,amsmath,amscd,amsthm,amsfonts,wasysym,mathrsfs,amsxtra,stmaryrd,verbatim}

\usepackage{graphicx,array,url}
\usepackage[vcentermath]{genyoungtabtikz}
\usepackage{tikz}
\usepackage[section]{placeins}
\usepackage{afterpage}
\input xy
\xyoption{all}

\newcommand{\RR}{\mathbb{R}}

\newcommand{\image}{\textnormal{im}\,}

\newcommand{\Hom}{\textnormal{Hom}}
\newcommand{\dimension}{\textnormal{dim}\,}

\newcommand{\rank}{\textnormal{rk}\,}

\newcommand{\Ext}{\textnormal{Ext}}

\newcommand{\Ac}{\mathcal{A}}
\newcommand{\Bc}{\mathcal{B}}
\newcommand{\Cc}{\mathcal{C}}
\newcommand{\Fc}{\mathcal{F}}
\newcommand{\Tc}{\mathcal{T}}

\newcommand{\Coh}{\mathrm{Coh}}

\newcommand{\arinj}{\ar@{^{(}->}}
\newcommand{\arsurj}{\ar@{->>}}
\newcommand{\areq}{\ar@{=}}

\newcommand{\wh}{\widehat}

\newcommand{\Bl}{\mathcal{B}^l}
\newcommand{\ch}{\mathrm{ch}}

\newcommand{\scalea}{\scalebox{0.5}}

\newcommand{\oo}{{\overline{\omega}}}

\newcommand{\whPhi}{{\wh{\Phi}}}

\begin{document}

\title{Fourier-Mukai transforms of slope stable torsion-free sheaves on Weierstrass elliptic surfaces}

\author[Jason Lo]{Jason Lo}
\address{Department of Mathematics \\
California State University Northridge\\
18111 Nordhoff Street\\
Northridge CA 91330 \\
USA}
\email{jason.lo@csun.edu}


\begin{abstract}
On a  Weierstra{\ss} elliptic surface $X$, we define a `limit' of Bridgeland stability conditions, denoted $Z^l$-stability, by varying the polarisation in the definition of Bridgeland stability along a curve in the ample cone of $X$.  We show that a slope stable torsion-free sheaf of positive (twisted) degree or a slope stable locally free sheaf is taken by a Fourier-Mukai transform on $D^b(X)$ to a $Z^l$-stable object, while a $Z^l$-semistable object of nonzero fiber degree can be modified so that its inverse Fourier-Mukai transform is a slope semistable torsion-free sheaf.
\end{abstract}

\maketitle
\tableofcontents

\section{Introduction}

Fourier-Mukai transforms on elliptic surfaces have been intensely studied over the years.  Understanding the image of a slope stable or Gieseker stable torsion-free sheaf under a Fourier-Mukai transform has been a major problem in this area and considered by numerous authors, in works such as \cite{BriTh,FMTes,bernardara2014,YosAS,YosPI,YosPII,FMNT}.  In this article, we give a fresh approach to this problem by interpreting the Fourier-Mukai transform of slope stability for sheaves as a `limit' of Bridgeland stability.

More precisely, recall that the construction of Bridgeland stability conditions depends on the choice of a polarisation $\omega$.  On  a Weierstra{\ss} surface $X$, by varying the polarisation $\omega$ along a curve in the ample cone, we define a `limit' of Bridgeland stability conditions, denoted as `$Z^l$-stability' in the article.    In our main theorem, Theorem \ref{thm:Lo14Thm5-analogue}, we show that on a Weierstra{\ss} elliptic surface $X$, if $E$ is a slope stable torsion-free sheaf of positive twisted degree or  a slope stable locally free sheaf, then the Fourier-Mukai transform of $E$ is a $Z^l$-stable object; on the other hand, if $F$ is a $Z^l$-semistable object of nonzero fiber degree, then $F$ has a  modification $F'$ where the inverse Fourier-Mukai transform of $F'$ is a slope semistable torsion-free sheaf.  The key premise of Theorem \ref{thm:Lo14Thm5-analogue} is that, in addition to assuming the sheaf $E$ is torsion-free, we do not fix the Chern character of $E$.  That is, we attempt to understand the action of the Fourier-Mukai transform on the very notion of slope stability for sheaves in general, rather than slope stability for sheaves of a specific Chern character.

After setting up the preliminaries and introducing the cohomological Fourier-Mukai transforms in Section \ref{sec:prelim}, we give the precise construction of $Z^l$-stability, considered as a limit of Bridgeland stability conditions, on a Weierstra{\ss} surface in Section \ref{sec:limitBridgelandconstr}.  We prove our main result, Theorem \ref{thm:Lo14Thm5-analogue}, in Section \ref{sec:maintheorem}, and verify the Harder-Narasimhan property of $Z^l$-stability in Section \ref{sec:HNproperty}.  We end this article with a connection between $Z^l$-stability and Bridgeland stability conditions in Section \ref{sec:BrivslimitBri}.

The essential ideas in this paper have appeared in the author's preceding works on elliptic threefolds \cite{Lo14,Lo15}.  In \cite{Lo14}, the author considered the product $X = C \times S$ of a smooth elliptic curve and a K3 surface $S$ of Picard rank 1.  By considering this product threefold as a trivial elliptic fibration over $S$ via the second projection, the author proved an analogue of Theorem \ref{thm:Lo14Thm5-analogue}, with `limit tilt stability' on a threefold playing the role of `limit Bridgeland stability' in this article.  The results in \cite{Lo14} were then generalised in the first half of \cite{Lo15} to a class of Weierstra{\ss} elliptic threefolds, where the base of the elliptic fibration could be a Fano surface, an Enriques surface or a K3 surface.  In the second half of \cite{Lo15}, the same ideas were carried out on the triangulated category of complexes that vanish on the generic fiber of the fibration, giving a notion of limit stability on this triangulated category.  It was shown in \cite{Lo15} that this limit stability corresponds to slope stability for 1-dimensional sheaves under a Fourier-Mukai transform.

The present article demonstrates that the construction in \cite{Lo14,Lo15} holds not only on elliptic threefolds but also on elliptic surfaces.

\paragraph[Acknowledgements] The author thanks Ching-Jui Lai,  Wanmin Liu and  Ziyu Zhang for many helpful discussions on elliptic fibrations.  He would also like to thank the National Center for Theoretical Sciences in Taipei for their support and hospitality during  December 2016-January 2017, when this work was initiated.  He also  thanks the Center for Geometry and Physics, Institute for Basic Science in Pohang, South Korea, for their support and hospitality throughout the author's visit in June-July 2017, during which much of this work was completed.

\section{Preliminaries}\label{sec:prelim}

\paragraph[Our elliptic fibration]  Throughout \label{para:ourellipticfibration} this article, unless otherwise stated, we will write $p : X \to B$ to denote an elliptic surface that is a Weierstra{\ss} fibration in the sense of \cite[Section 6.2]{FMNT} where $X$ is a smooth projective surface and $B$ is a smooth projective curve.  In particular, the morphism $p$ is  flat  with all the fibers  geometrically integral,  and $p$ admits a section $\sigma : B \hookrightarrow X$ with image $\Theta = \sigma (B)$ that does not intersect any singular point of the singular fibers.  We do not place any  restriction on the Picard rank of $X$.

\paragraph[Notation] We collect here preliminary notions and notations that will be used throughout the article.

\subparagraph[Twisted Chern character] For any divisor $B$ on a smooth projective  surface $X$ and any $E \in D^b(X)$, the twisted Chern character $\ch^B(E)$ is defined as
\[
\ch^B(E) = e^{-B}\ch(E)
= (1-B+\tfrac{B^2}{2})\ch(E).
\]
  We write $\ch^B(E) = \sum_{i=0}^2 \ch_i^B(E)$ where
\begin{align*}
  \ch_0^B(E) &= \ch_0(E) \\
  \ch_1^B(E) &= \ch_1(E)-B\ch_0(E) \\
  \ch_2^B(E) &= \ch_2(E) -B\ch_1(E)+\tfrac{B^2}{2}\ch_0(E)
\end{align*}
We sometimes refer to the divisor $B$ involved in the twisting of the Chern character as the `$B$-field'.  In this article, there should be no risk of confusion as to whether $B$ refers to the base of the elliptic fibration $p$ or a $B$-field.

\subparagraph[Cohomology] Suppose $\Ac$ is an abelian category and $\Bc$ is the heart of a t-structure on $D^b(\Ac)$.  For any object $E \in D^b(\Ac)$, we will write $\mathcal{H}^i_{\Bc}(E)$ to denote the $i$-th cohomology object of $E$ with respect to the t-structure with heart $\Bc$.  When $\Bc = \Ac$, i.e.\ when the aforementioned t-structure is the standard t-structure on $D^b(\Ac)$, we will write $H^i(E)$ instead of $\mathcal{H}^i_{\Ac}(E)$.

Given a smooth projective variety $X$, the dimension of an object $E \in D^b(X)$ will be denoted by $\dimension E$, and refers to the dimension of its support, i.e.\
\[
  \dimension E = \dimension \bigcup_i \mathrm{supp}\, H^i(E).
\]
That is, for a coherent sheaf $E$, we have $\dimension E = \dimension \mathrm{supp} (E)$.

\subparagraph[Torsion pairs and tilting]  A \label{para:torsionpairtilting} torsion pair $(\Tc, \Fc)$ in an abelian category $\mathcal{A}$ is a pair of full subcategories $\Tc, \Fc$ such that
\begin{itemize}
\item[(i)] $\Hom_{\Ac}(E', E'')=0$ for all $E' \in \Tc, E'' \in \Fc$.
\item[(ii)] Every object $E \in \Ac$ fits in an $\Ac$-short exact sequence
\[
0 \to E' \to E \to E'' \to 0
\]
for some $E' \in \Tc, E'' \in \Fc$.
\end{itemize}
The decomposition of $E$ in (ii) is canonical \cite[Chapter 1]{HRS}, and we will occasionally refer to it as the $(\Tc,\Fc)$-decomposition of $E$ in $\Ac$. Whenever we have a torsion pair $(\Tc, \Fc)$ in an abelian category $\mathcal{A}$, we will refer to $\Tc$ (resp.\ $\Fc$) as the torsion class (resp.\ torsion-free class) of the torsion pair.  The extension closure in $D^b(\Ac)$
\[
  \Ac' = \langle \Fc [1], \Tc \rangle
\]
is the heart of a t-structure on $D^b(\Ac)$, and hence an abelian subcategory of $D^b(\Ac)$.  We call $\Ac'$ the tilt of $\Ac$ at the torsion pair $(\Tc, \Fc)$.  More specifically, the category $\Ac'$ is the heart of the t-structure $(D^{\leq 0}_{\Ac'}, D^{\geq 0}_{\Ac'})$ on $D^b(\Ac)$ where
\begin{align*}
  D^{\leq 0}_{\Ac'} &= \{ E \in D^b(\Ac) : \mathcal{H}_{\Ac}^0 (E)\in \Tc, \mathcal{H}^i_{\Ac} (E)= 0 \, \forall\, i > 0 \}, \\
  D^{\geq 0}_{\Ac'} &= \{ E \in D^b(\Ac) : \mathcal{H}_{\Ac}^{-1} (E)\in \Fc, \mathcal{H}^i_{\Ac} (E)= 0 \,\forall\, i < -1 \}.
\end{align*}
A subcategory of $\Ac$ will be called a torsion class (resp.\ torsion-free class) if it is the torsion class (resp.\ torsion-free class) in some torsion pair in $\Ac$.  By a lemma of Polishchuk \cite[Lemma 1.1.3]{Pol}, if $\Ac$ is a noetherian abelian category, then every subcategory that is closed under extension and quotient in $\Ac$ is a torsion class in $\Ac$.

For any subcategory $\mathcal{C}$ of an abelian category $\mathcal{A}$, we will set
\[
  \mathcal{C}^\circ = \{ E \in \mathcal{A} : \Hom_{\mathcal{A}}(F,E)=0 \text{ for all } F \in \mathcal{C} \}
\]
when $\mathcal{A}$ is clear from the context.  Note that whenever $\mathcal{A}$ is noetherian and $\mathcal{C}$ is closed under extension and quotient in $\mathcal{A}$, the pair $(\mathcal{C},\mathcal{C}^\circ)$ gives a torsion pair in $\mathcal{A}$.


\subparagraph[Torsion $n$-tuples] A torsion $n$-tuple $(\Cc_1, \Cc_2,\cdots,\Cc_n)$ in an abelian category $\Ac$ as defined in \cite[Section 2.2]{Pol2} is a collection of full subcategories of $\Ac$ such that
\begin{itemize}
\item $\Hom_{\Ac} (C_i,C_j)=0$ for any $C_i \in \mathcal C_i, C_j \in \mathcal C_j$ where $i<j$.
\item Every object $E$ of $\Ac$ admits a filtration in $\Ac$
\[
  0=E_0 \subseteq E_1 \subseteq E_2 \subseteq \cdots \subseteq E_n = E
\]
where $E_i/E_{i-1} \in \mathcal C_i$ for each $1 \leq i \leq n$
\end{itemize}
(see also \cite[Definition 3.5]{Toda2}).   Note that, given a torsion $n$-tuple in $\Ac$ as above, the pair
\[
(\langle \Cc_1, \cdots, \Cc_i\rangle, \langle \Cc_{i+1}, \cdots, \Cc_n \rangle )
\]
 is a torsion pair in $\Ac$ for any $1 \leq i \leq n-1$.

\subparagraph[Fourier-Mukai transforms] For any Weierstra{\ss} elliptic fibration $p : X \to B$ in the sense of \cite[Section 6.2]{FMNT} where $X$ is smooth, there is a pair of relative Fourier-Mukai transforms $\Phi, \whPhi : D^b(X) \overset{\thicksim}{\to} D^b(X)$ whose kernels are both sheaves on $X \times _B X$, satisfying
\begin{equation}\label{eq:PhiwhPhiisidshifted}
  \whPhi \Phi = \mathrm{id}_{D^b(X)}[-1] = \Phi \whPhi.
\end{equation}
In particular, the kernel of $\Phi$ is the relative Poincar\'{e} sheaf for the fibration $p$, which is a universal sheaf for the moduli problem that parametrises degree-zero, rank-one torsion-free sheaves on the fibers of $p$.  An object $E \in D^b(X)$ is said to be $\Phi$-WIT$_i$ if $\Phi E$ is a coherent sheaf sitting at degree $i$.  In this case, we write $\wh{E}$ to denote the coherent sheaf satisfying  $\Phi E \cong \wh{E} [-i]$ up to isomorphism.  The notion of $\whPhi$-WIT$_i$ can similarly be defined.  The identities \eqref{eq:PhiwhPhiisidshifted} imply that, if a coherent sheaf $E$ on $X$ is $\Phi$-WIT$_i$ for $i=0, 1$, then $\wh{E}$ is $\whPhi$-WIT$_{1-i}$.  For $i=0,1$, we will define the category
\[
  W_{i,\Phi} = \{ E \in \Coh (X) : E \text{ is $\Phi$-WIT$_i$} \}
\]
and similarly for $\whPhi$.  Due to the symmetry between $\Phi$ and $\whPhi$, the properties held by $\Phi$ also hold for $\whPhi$.     See \cite[Section 6.2]{FMNT} for more background on the functors $\Phi, \whPhi$.

\subparagraph[Subcategories of $\Coh (X)$]  Let $p : X \to B$ be an elliptic surface as in \ref{para:ourellipticfibration}. For any integers $d \geq e$, we set
\begin{align*}
\Coh^{\leq d}(X) &= \{ E \in \Coh (X): \dimension \mathrm{supp}(E) \leq d\} \\
\Coh^d(p)_e &= \{ E \in \Coh (X): \dimension \mathrm{supp}(E) = d, \dimension p (\mathrm{supp}(E)) = e\} \\
\{ \Coh^{\leq 0} \}^\uparrow &= \{ E \in \Coh (X): E|_b \in \Coh^{\leq 0} (X_b) \text{ for all closed points $b \in B$} \}
\end{align*}
where $\Coh^{\leq 0}(X_b)$ is the category of coherent sheaves supported in dimension 0 on the fiber $p^{-1}(b)=X_b$, for the closed point $b \in B$.  We will refer to coherent sheaves that are supported on a finite number of fibers of $p$ as fiber sheaves.  Adopting the notation in \cite[Section 3]{Lo14}, we also define
\begin{align*}
   \scalea{\gyoung(;;,;;+)}&:= \Coh^{\leq 0}(X) \\
    \scalea{\gyoung(;;+,;;+)} &:= \{ E \in \Coh^1(\pi)_0 :  \text{ all $\mu$-HN factors of $E$ have $\infty>\mu>0$}\} \\
  \scalea{\gyoung(;;+,;;0)} &:=\{ E \in \Coh^1(\pi)_0 : \text{ all $\mu$-HN factors of $E$ have $\mu=0$}\} \\
  \scalea{\gyoung(;;+,;;-)} &:= \{ E \in \Coh^1(\pi)_0 :  \text{ all $\mu$-HN factors of $E$ have $\mu<0$}\} \\
   \scalea{\gyoung(;;*,;+;*)} &:= \Coh^1(\pi)_1 \cap \{\Coh^{\leq 0}\}^\uparrow\\
\scalea{\gyoung(;+;*,;+;*)} &= \{ E \in W_{0,\whPhi}: \dimension E = 2\} \\
\scalea{\gyoung(;+;*,;0;*)} &= \{ E \in \Phi (\{\Coh^{\leq 0}\}^\uparrow \cap \Coh^{\leq 1}(X)) : \dimension E = 2\} \\
\scalea{\gyoung(;+;*,;-;*)} &= \{ E \in  W_{1,\wh{\Phi}} : \dimension E =2, f\ch_1(E)\neq 0 \} .
\end{align*}
Note that the definitions of $\scalea{\gyoung(;+;*,;+;*)}, \scalea{\gyoung(;+;*,;0;*)}$ and $\scalea{\gyoung(;+;*,;-;*)}$ depend on the Fourier-Mukai functor $\whPhi$.  We will use the same notation to denote the corresponding category defined using $\whPhi$; it will always be clear from the context which Fourier-Mukai functor the definition is with respect to.  The Fourier-Mukai transform $\Phi$ induces the following equivalences, as already observed in \cite[Remark 3.1]{Lo14}:
\[
  \xymatrix @-1.3pc{
\scalea{\gyoung(;;,;;+)} \ar[dr] & \scalea{\gyoung(;;+,;;+)} \ar@/^1.5pc/[dd] & \scalea{\gyoung(;;*,;+;*)} \ar[dr] & \scalea{\gyoung(;+;*,;+;*)}  \ar@/^1.5pc/[dd]  \\
 & \scalea{\gyoung(;;+,;;0)} & & \scalea{\gyoung(;+;*,;0;*)} \\
  & \scalea{\gyoung(;;+,;;-)} & & \scalea{\gyoung(;+;*,;-;*)}
 }
\]
A concatenation of more than one such diagram will mean the extension closure of the categories involved; for example, the concatenation
\[
\xymatrix @-2.3pc{
\scalea{\gyoung(;;,;;+)} & \scalea{\gyoung(;;+,;;+)} \\
& \scalea{\gyoung(;;+,;;0)}
}
\]
is the extension closure of all slope semistable fiber sheaves of slope at least zero (including sheaves supported in dimension zero, which are slope semistable fiber sheaves of slope $+\infty$).

The category $\Coh^{\leq d}(X)$ for any integer $d \geq 0$, as well as $\{\Coh^{\leq 0}\}^\uparrow$ and $W_{0,\whPhi}$ are all torsion classes in $\Coh (X)$.  From \ref{para:torsionpairtilting}, each of these torsion classes determines a tilt of $\Coh (X)$, and hence determines a t-structure on $D^b(X)$. For instance, we have the torsion pairs $(W_{0,\whPhi}, W_{1,\whPhi})$ and $(\Coh^{\leq d}(X), \Coh^{\geq d+1}(X))$ in $\Coh (X)$.

\subparagraph[Slope-like functions]  Suppose \label{para:slopelikefunctions} $\Ac$ is an abelian category.  We call a function $\mu$ on $\Ac$ a slope-like function if $\mu$ is defined by
\[
  \mu (F) = \begin{cases} \frac{C_1(F)}{C_0(F)} &\text{ if $C_0(F) \neq 0$} \\
  +\infty &\text{ if $C_0(F)=0$} \end{cases}
\]
where $C_0, C_1 : K(\Ac) \to \mathbb{Z}$ are a pair of group homomorphisms  satisfying: (i) $C_0(F) \geq 0 $ for any $F \in \Ac$; (ii) if $F \in \Ac$ satisfies $C_0(F)=0$, then $C_1 (F) \geq 0$.  The additive group $\mathbb{Z}$ in the definition of a slope-like function can be replaced by any  discrete additive subgroup of $\mathbb{R}$.  Whenever $\Ac$ is a noetherian abelian category, every slope-like function possesses the Harder-Narasimhan property \cite[Section 3.2]{LZ2}; we will then say an object $F \in \Ac$ is $\mu$-stable (resp.\ $\mu$-semistable) if, for every short exact sequence $0 \to M \to F \to N \to 0$ in $\Ac$ where $M,N \neq 0$, we have $\mu (M) < (\text{resp.} \leq \, )\, \mu (N)$.

\subparagraph[Slope stability] Suppose \label{para:muomegaBslopefctndefn} $X$ is a smooth projective surface with a fixed ample divisor   $\omega$ and a fixed divisor $B$.  For any coherent sheaf $E$ on $X$, we  define
\[
  \mu_{\omega,B} (E) = \begin{cases}
  \frac{\omega \ch_1^B (E)}{\ch_0^B(E)} &\text{ if $\ch_0^B(E) \neq 0$} \\
  +\infty &\text{ if $\ch_0^B(E)=0$}\end{cases}.
\]
A coherent sheaf $E$ on $X$ is said to be $\mu_{\omega,B}$-stable or slope stable (resp.\ $\mu_{\omega,B}$-semistable or slope semistable) if, for every short exact sequence in $\Coh (X)$ of the form
\[
0 \to M \to E \to N \to 0
\]
where $M,N \neq 0$, we have $\mu_{\omega,B} (M) < (\text{resp.}\, \leq) \, \mu_{\omega,B} (N)$.  Note that for any  coherent sheaf $M$ on $X$ with $\ch_0(M) \neq 0$,  we have
\[
\mu_{\omega,B} (M) = \frac{\omega \ch_1^B(M)}{\ch_0(M)} = \frac{\omega\ch_1(M)-\omega B\ch_0(M)}{\ch_0(M)} = \mu_\omega (M) -\omega B.
\]
Hence $\mu_{\omega,B}$-stability is equivalent to $\mu_\omega$-stability for  coherent sheaves.  When $B=0$, we often write $\mu_\omega$ for $\mu_{\omega,B}$.

\subparagraph[Bridgeland stability conditions on surfaces] Suppose \label{para:Bridgestabonsurfaces} $X$ is a smooth projective surface.  For any ample divisor $\omega$ on $X$, we can define the following subcategories of $\Coh (X)$
\begin{align*}
  \Tc_\omega &= \langle E \in \Coh (X) :  E \text{ is $\mu_\omega$-semistable}, \, \mu_\omega (E) > 0\rangle, \\
  \Fc_\omega &= \langle E \in \Coh (X) : E \text{ is $\mu_\omega$-semistable}, \, \mu_\omega (E) \leq 0 \rangle.
\end{align*}
Since the slope function $\mu_\omega$ has the Harder-Narasimhan property, the pair $(\Tc_\omega,\Fc_\omega)$ is a torsion pair in $\Coh (X)$.  The extension closure
\[
  \Bc_\omega = \langle \Fc [1], \Tc \rangle
\]
in $D^b(X)$ is thus  a tilt of the heart $\Coh (X)$, i.e.\ $\Bc_\omega$ is  the heart of a bounded t-structure on $D^b(X)$ and is an abelian subcategory of $D^b(X)$.   If we set
\begin{equation}\label{eq:Zwformula}
  Z_\omega (F) = -\int_X e^{-i\omega}\ch(F) = -\ch_2(F) + \tfrac{\omega^2}{2}\ch_0(F) + i\omega \ch_1(F),
\end{equation}
then the pair $(\Bc_\omega, Z_\omega)$ gives a Bridgeland stability condition on $D^b(X)$, as shown by  Arcara-Bertram  in \cite{ABL}.    In particular, for any nonzero object $F \in \Bc_\omega$, the complex number $Z_\omega (F)$ lies in the upper-half complex plane (that includes the negative real axis)
\[
  \mathbb{H} = \{ re^{i\pi \phi} : r>0, \phi \in (0,1] \}.
\]
This allows us to define the phase $\phi (F)$ of any nonzero object $F \in \Bc_\omega$ using the relation
\[
  Z_\omega (F) \in \mathbb{R}_{>0}\cdot e^{i \pi \phi (F)}  \text{\quad where $\phi (F) \in (0,1]$}.
\]
We then say an object $F \in \Bc_\omega$ is $Z_\omega$-stable (resp.\ $Z_\omega$-semistable) if, for all $\Bc_\omega$-short exact sequence
\[
0 \to M \to F \to N \to 0
\]
where $M, N \neq 0$, we have $\phi (M) < \phi (N)$ (resp.\ $\phi (M) \leq \phi (N)$).

\paragraph[The N\'{e}ron-Severi group $\mathrm{NS}(X)$] Since \label{para:NSX} our elliptic fibration $p$ is assumed to be Weierstra{\ss}, there exists a section, and the Picard rank of $X$ is finite by the Shioda-Tate formula \cite[VII 2.4]{MirLec}. Also,  the N\'{e}ron-Severi group $\mathrm{NS}(X)$ is generated by the fiber class $f$ and a finite number of sections $\Theta_0 := \Theta, \Theta_1, \cdots, \Theta_r$ for some $r \geq 0$ \cite[VII 2.1]{MirLec}.

\paragraph[The cohomological Fourier-Mukai transforms] For  \label{para:cohomFMTformulas} any $E \in D^b(X)$, if we let
\begin{align}
  n &= \ch_0(E), \notag\\
  d = f\ch_1(E)&, \,\,  c = \Theta\ch_1(E), \notag\\
  s &= \ch_2(E), \label{eq:chEnotation}
\end{align}
then from the cohomological Fourier-Mukai transform in   \cite[(6.21)]{FMNT} we have
\begin{align}
\ch_0 (\Phi E) &= d, \notag\\
\ch_1 (\Phi E) &= -\ch_1(E) + dp^\ast \bar{K} + (d-n)\Theta + (c-\tfrac{1}{2}ed + s)f, \notag\\
\ch_2 (\Phi E) &= (-c-de + \tfrac{1}{2}ne). \label{eq:cohomFMT-Phi}
\end{align}
where $\Theta^2 = -e$ and $\bar{K} = c_1 (p_\ast \omega_{X/B})$.  Since $p^\ast \bar{K} \equiv ef$, we have $ \ch_1(\Phi E)\cdot f =  -n$ and $\ch_1(\Phi E)\cdot \Theta =  (s-\tfrac{e}{2}d)+ne$.
In particular, for any $m \in \mathbb{R}$ we have
\begin{align}
\ch_1(\Phi E)\cdot f &=  -n,\notag\\
\ch_1(\Phi E)\cdot (\Theta + mf) &= s-\tfrac{e}{2}d + (e-m)n. \label{eq:cohomFMT-Phi-ch1}
\end{align}  On the other hand, from \cite[(6.22)]{FMNT} we have
\begin{align}
  \ch_0 (\wh{\Phi}E) &= d \notag\\
  \ch_1 (\wh{\Phi}E) &= \ch_1(E) - np^\ast \bar{K} - (d+n)\Theta + (s+en-c-\frac{e}{2}d)f, \notag \\
  \ch_2 (\wh{\Phi}(E)) &= -(c+de+\tfrac{e}{2}n). \label{eq::cohomFMT-hatPhi}
\end{align}
This gives $\ch_1 (\wh{\Phi} E)\cdot f = -n$ and $\ch_1 (\wh{\Phi}E)\cdot \Theta = s+\tfrac{e}{2}d + ne$.  In particular, for any $m \in \mathbb{R}$ we have
 \begin{align}
  \ch_1 (\wh{\Phi} E)\cdot f &= -n, \notag\\
  \ch_1 (\wh{\Phi} E) \cdot (\Theta + mf) &= s+\tfrac{e}{2}d+(e-m)n.  \label{eq::cohomFMT-hatPhi-ch1}
\end{align}

\paragraph[Some intersection numbers] Here \label{para:basicintersectionnumbers} we collect some intersection numbers that will be used throughout the rest of the paper.  For any $m \in \RR$ we have
\begin{equation*}
  (\Theta +mf)^2 = \Theta^2 + 2m = 2m-e.
\end{equation*}
Recall that for  any section $\Theta$ of the fibration $p$, the divisor $\Theta + mf$ on $X$ is  ample for $m \gg 0$ \cite[Proposition 1.45]{KM}. We will often work with a polarisation of the form
\begin{equation}\label{eq:omeganotation}
  \omega = u(\Theta + mf) + vf
\end{equation}
for some $u, v \in \mathbb{R}$, which gives
\begin{equation*}
  \tfrac{\omega^2}{2} = (m-\tfrac{e}{2})u^2 + uv.
\end{equation*}
As a result, if we use the notation for $\ch(E)$ in \eqref{eq:chEnotation} then $(\Theta + mf) \ch_1 (E) = c + md$ and
\begin{align*}
\omega \ch_1 (E) &=  (u(\Theta + mf)+vf)\ch_1(E) \\
&= uc + (um+v)d.
\end{align*}
If we also set
\[
 \oo = a(\Theta + mf)+bf
\]
 where $a, b \in \RR$ and fix $B = \tfrac{e}{2}f$ then
\begin{align*}
  \oo \ch_1^B(E) &=  \oo (\ch_1(E) - \tfrac{e}{2}f\ch_0(E)) \\
  &= a(c-\tfrac{e}{2}n) + (am+b)d.
\end{align*}
As a result, when $\oo$ is an ample divisor on $X$, we can write the twisted slope function $\mu_{\oo,B}$ as
\begin{equation}\label{eq:muBbbE}
  \mu_{\oo,B}(E) = \tfrac{1}{n} (a(c-\tfrac{e}{2}n) + (am+b)d).
\end{equation}
On the other hand, when $\omega$ is an ample divisor on $X$, with respect to the central charge \eqref{eq:Zwformula} we have
\begin{align}
  Z_\omega (\Phi E [1]) &=  \ch_2(\Phi E)  -\tfrac{\omega^2}{2}\ch_0(\Phi E) - i\omega \ch_1(\Phi E) \notag\\
  &= (-c-de+\tfrac{e}{2}n) - ((m-\tfrac{e}{2})u^2+uv)d - i \left( u ( s-\tfrac{e}{2}d+(e-m)n) - vn \right)  \text{from \ref{eq:cohomFMT-Phi-ch1}} \notag\\
  &= (-c + \tfrac{e}{2}n) - ( (m-\tfrac{e}{2})u^2 + uv +e)d + i \left( u (-( s-\tfrac{e}{2}d)+(m-e)n)+vn\right). \label{eq:ZwPhiEshift}
\end{align}

\paragraph[Heuristics] Comparing \label{para:heuristics} the coefficients of the characteristic classes $(c-\tfrac{e}{2}n)$ and $d$ in the  expressions for $\mu_{\oo,B} (E)$ and $Z_\omega (\Phi E [1])$, we see that for fixed $m, a, b>0$, if $v \to \infty$ along the curve
\[
  \frac{am+b}{a} = (m-\tfrac{e}{2})u^2+uv+e
\]
i.e.\
\begin{equation*}
  m + \tfrac{b}{a} = (m-\tfrac{e}{2})u^2 + uv +e
\end{equation*}
then $\oo \ch_1^B(E)$ is a negative scalar multiple of $\Re Z_\omega (\Phi E [1])$, while $\Im Z_\omega (\Phi E [1])$ is dominated by a positive scalar multiple of $\ch_0(E)$.  This suggests that  for $v \gg 0$, $\mu_{\oo,B}$-stability for $E$ should be an `approximation' of  $Z_\omega$-stability up to the Fourier-Mukai transform $\Phi$, or that $Z_\omega$-stability is a `refinement' of $\mu_{\oo,B}$-stability for $E$ up to $\Phi$.  We will make this idea precise in the remainder of this paper.  The computation above also motivates us to consider the change of variables
\begin{equation*}
  \lambda = b, \,\,\alpha = \tfrac{b}{a}
\end{equation*}
so that $\oo$ can be written as
\begin{equation}\label{eq:oonotation}
  \oo = \tfrac{\lambda}{\alpha} (\Theta + mf) + \lambda f.
\end{equation}
Then  $\mu_{\oo,B}$-stability depends only on $\alpha$ and not $\lambda$, and we  can think of $\mu_{\oo,B}$-stability as being approximated by $Z_\omega$-stability as $v \to \infty$ along the curve
\begin{equation}\label{eq:vhyperbolaequation}
 m + \alpha = (m-\tfrac{e}{2})u^2 + uv +e.
\end{equation}

\paragraph[Decomposing $\mu_\omega$] Suppose $F$ is an object in $D^b(X)$.  With $\omega$ as in \eqref{eq:omeganotation},  we can rewrite $\mu_\omega (F)$ as
\begin{align}
  \mu_\omega (F) = \frac{\omega \ch_1(F)}{\ch_0(F)} &= u \frac{(\Theta + mf)\ch_1(F)}{\ch_0(F)}+ v \frac{f\ch_1(F)}{\ch_0(F)} \notag\\
  &= u \mu_{\Theta + mf}(F) + v \mu_f (F). \label{eq:muwdecomposition}
\end{align}
Recall that the divisor $\Theta + mf$ is ample on $X$ for $m \gg 0$ while $f$ is a nef divisor on $X$. Therefore,  both  $\mu_{\Theta + mf}$ and $\mu_f$ are `slope-like' functions with the Harder-Narasimhan property (see \ref{para:slopelikefunctions}).

\paragraph For \label{para:ch1realZwcomparison} fixed $\lambda, \alpha >0$, with  $\oo$ as in \eqref{eq:oonotation}, $\omega$ as in \eqref{eq:omeganotation}, and $u,v>0$ under the constraint \eqref{eq:vhyperbolaequation}, we have the following observation that will be useful later on: with the same notation for $\ch(E)$ as in \ref{para:basicintersectionnumbers}, for the $B$-field $B=\tfrac{e}{2}f$ we have
\begin{align}
  \oo \ch_1^B (E) &= \oo (\ch_1(E) - B \ch_0(E)) \notag\\
  &= \tfrac{\lambda}{\alpha} ( (c-\tfrac{e}{2}n) + (m + \alpha)d) \notag\\
  &= - \tfrac{\lambda}{\alpha} \Re Z_\omega (\Phi E [1]). \label{eq:ooch1twistReZePhiE1relation}
\end{align}
In particular, if $F$ is a $\whPhi$-WIT$_1$ sheaf on $X$ of nonzero rank  with $f\ch_1(F)=0$, then $\wh{F} = \whPhi F[1]$ is a sheaf supported in dimension 1, implying $\oo \ch_1^B(\wh{F}) = \oo \ch_1 (\wh{F})>0$.  Then
\[
  \Re Z_\omega (F) = \Re Z_\omega (\Phi \wh{F})=-\Re Z_\omega (\Phi \wh{F}[1])= \tfrac{\alpha}{\lambda}\oo \ch_1 (\wh{F})>0.
\]


\section{Constructing a limit Bridgeland stability}\label{sec:limitBridgelandconstr}

Since the Bridgeland stability condition $(\Bc_\omega, Z_\omega)$ on $X$ depends on $\omega$, varying $\omega$ will change the stability condition accordingly (see \ref{para:Bridgestabonsurfaces}).  In this section, we will show that when $\omega$ is written in the form
\[
\omega = u(\Theta + mf) + vf
\]
and $v \to \infty$ subject to the constraint \eqref{eq:vhyperbolaequation}, we obtain a notion of stability with the Harder-Narasimhan property, which can be considered as a `limit Bridgeland stability'.

Due to the symmetry between $\Phi$ and $\whPhi$, all the results involving $\Phi$ and $\whPhi$ in this section and beyond still hold if we interchange $\Phi$ and $\whPhi$ (except for explicit computations involving Chern classes, since the cohomological Fourier-Mukai transforms corresponding to $\Phi$ and $\whPhi$ are different - see \ref{para:cohomFMTformulas}).

For the rest of this section, let us fix an $m >0$ so that $\Theta + kf$ is ample for all $k \geq m$.  We will write $\omega$ in the form \eqref{eq:omeganotation} with $u,v>0$.


\begin{lem}\label{lem:TclFcldefinitions}
Suppose $u_0>0$ and $F \in \Coh (X)$.
\begin{itemize}
\item[(1)] The following are equivalent:
  \begin{itemize}
  \item[(a)] There exists $v_0>0$ such that $F \in \Fc_\omega$ for all $(v,u) \in (v_0,\infty) \times (0,u_0)$.
  \item[(b)] There exists $v_0>0$ such that, for every nonzero subsheaf $A \subseteq F$, we have $\mu_\omega (A) \leq 0$ for all $(v,u) \in (v_0,\infty)\times (0,u_0)$.
  \item[(c)] For every nonzero subsheaf $A \subseteq F$, either (i) $\mu_f (A) < 0$, or (ii) $\mu_f (A)=0$ and also $\mu_{\Theta + mf}(A) \leq 0$.
  \end{itemize}
\item[(2)] The following are equivalent:
  \begin{itemize}
  \item[(a)] There exists $v_0>0$ such that $F \in \Tc_\omega$ for all $(v,u) \in (v_0,\infty) \times (0,u_0)$.
  \item[(b)]  There exists $v_0>0$ such that, for every nonzero sheaf quotient $F \twoheadrightarrow A$, we have $\mu_\omega (A) > 0$ for all $(v,u) \in (v_0,\infty) \times (0,u_0)$.
  \item[(c)] For any nonzero sheaf quotient $F \twoheadrightarrow A$, either (i) $\mu_f (A)>0$, or (ii) $\mu_f(A)=0$ and $\mu_{\Theta + mf}(A) >0$.
  \end{itemize}
\end{itemize}
\end{lem}

\begin{proof}
The proofs for parts (1) and (2) are essentially the same as those for \cite[Lemma 4.1]{Lo14} and \cite[Lemma 4.3]{Lo14}, respectively, if we replace the slope-like function $\mu^\ast$ in those proofs by $\mu_{\Theta + mf}$.
\end{proof}

\paragraph[A limit of the heart $\Bc_\omega$] We \label{para:TlFlproperties} now define the following subcategories of $\Coh (X)$:
\begin{itemize}
\item $\Tc^l$, the extension closure of all coherent sheaves satisfying condition (2)(c) in Lemma \ref{lem:TclFcldefinitions}.
\item $\Fc^l$,  the extension closure of all coherent sheaves satisfying condition (1)(c) in Lemma \ref{lem:TclFcldefinitions}.
\end{itemize}
We also define the extension closure in $D^b(X)$
\[
  \Bc^l = \langle \Fc^l [1], \Tc^l \rangle.
\]
Following an argument as in the proof of Lemma \ref{lem:TclFcldefinitions}, it is easy to check that the categories $\Tc^l, \Fc^l$ can equivalently be defined as
\begin{align*}
  \Tc^l &= \{ F \in \Coh (X): F \in \Tc_\omega \text{ for all $v \gg 0$ along \eqref{eq:vhyperbolaequation} } \} \\
  \Fc^l &= \{ F \in \Coh (X): F \in \Fc_\omega \text{ for all $v \gg 0$ along \eqref{eq:vhyperbolaequation} }\}.
\end{align*}
The following immediate properties  are analogous to those  in \cite[Remark 4.4]{Lo14}:
\begin{itemize}
\item[(i)] $\Coh^{\leq 1}(X) \subset \Tc^l$ since  all the torsion sheaves are contained in $\Tc_\omega$, for any ample divisor $\omega$.
\item[(ii)] $\Fc^l \subset \Coh^{=2}(X)$ since every object in $\Fc_\omega$ is a torsion-free sheaf, for any ample divisor $\omega$.
\item[(iii)] $W_{0,\whPhi} \subset \Tc^l$ by the same argument as in \cite[Remark 4.4(iii)]{Lo14}.
\item[(iv)] $f\ch_1(F) \geq 0$ for every $F \in \Bc^l$.  This is clear from the definition of $\Bc^l$ and Lemma \ref{lem:TclFcldefinitions}.  Lemma \ref{lem:TlFlaretorsionpair} below shows that $\Bc^l$ is the heart of a t-structure on $D^b(X)$, and hence an abelian category.  The subcategory
    \[
    \Bl_0 := \{ F \in \Bl : f\ch_1(F)=0 \}
    \]
    is then a Serre subcategory of $\Bl$.
\item[(v)] $\Fc^l \subset W_{1,\whPhi}$.  This follows from (iii) and Lemma \ref{lem:TlFlaretorsionpair} below.
\end{itemize}

\begin{lem}\label{lem:TlFlaretorsionpair}
The pair $(\Tc^l, \Fc^l)$ forms a torsion pair in $\Coh (X)$, and the category $\Bc^l$ is the heart of a bounded t-structure on $D^b(X)$.
\end{lem}

\begin{proof}
By \cite[Lemma 2.5]{Lo7}, we have
\[
  \begin{cases}
  f\ch_1(F) \geq 0 &\text{ if $F \in W_{0,\whPhi}$} \\
  f\ch_1(F) \leq 0 &\text{ if $F \in W_{1,\whPhi}$}
  \end{cases}.
\]
Armed with this observation,  the  argument  in the proof of \cite[Lemma 4.6]{Lo14} applies if we replace  $\mu^\ast$ by $\mu_{\Theta + mf}$ in that proof.
\end{proof}

Let $\mathbb{H}$ denote the strict upper-half complex plane (including the negative real axis)
\[
  \mathbb{H} = \{ re^{i\pi \phi} : r>0, \phi \in (0,1] \}.
\]

\begin{lem}\label{lem:ZwtakesBwupperplane}
Fix any $\alpha >0$.  For any nonzero $F \in \Bc^l$, we have $Z_\omega (F) \in \mathbb{H}$ as $v \to \infty$ along the curve \eqref{eq:vhyperbolaequation}.
\end{lem}

\begin{proof}
Part of the proof of $(\Bc_\omega, Z_\omega)$ being a Bridgeland stability condition on $D^b(X)$ \cite[Corollary 2.1]{ABL} asserts that $Z_\omega (F) \in \mathbb{H}$ for any nonzero object $F \in \Bc_\omega$.  This lemma thus follows from the characterisations of $\Tc^l, \Fc^l$ in Lemma \ref{lem:TclFcldefinitions}.
\end{proof}

\paragraph[$Z^l$-stability, limit Bridgeland stability] We can now define a `limit Bridgeland stability' as follows.  By Lemma \ref{lem:ZwtakesBwupperplane}, for any nonzero object $F \in \Bl$ we know that $Z_\omega (F)$ lies in the upper half plane $\mathbb{H}$ for $v \gg 0$ subject to the constraint \eqref{eq:vhyperbolaequation}, i.e.\
\[
 m + \alpha = (m-\tfrac{e}{2})u^2 + uv +e.
\]
We can then define a function germ $\phi (F) : \RR \to (0,1]$ for $v \gg 0$ via the relation
\[
  Z_\omega (F) \in \mathbb{R}_{>0}\cdot e^{i \pi \phi (F)(v)} \text{\quad for $v \gg 0$}.
\]
Although $u$ is only an implicit function in $v$ under the constraint \eqref{eq:vhyperbolaequation}, by requiring $u>0$ we can write $u$ as a function in $v$ for $v \gg 0$, in which case $O(u)=O(\tfrac{1}{v})$ as $v \to \infty$.  As a result, as $v \to \infty$, the function  $Z_\omega (F)$  is asymptotically equivalent to  a Laurent polynomial in $v$ over $\mathbb{C}$, allowing us to define a notion of stability as in the case of Bayer's polynomial stability \cite{BayerPBSC}: We say $F$ is $Z^l$-stable (resp.\ $Z^l$-semistable) if, for every $\Bl$-short exact sequence
\[
  0 \to M \to F \to N \to 0
\]
where $M,N \neq 0$, we have
\[
  \phi (M) < \phi (N) \text{ for $v \gg 0$}
\]
(resp.\ $\phi (M) \leq \phi (N)$ for $v \gg 0$).  We will usually write $\phi (M) \prec \phi (N)$ (resp.\ $\phi (M) \preceq \phi (N)$) to mean $\phi (M) < \phi (N)$ for $v \gg 0$ (resp.\ $\phi (M) \leq \phi (N)$ for $v \gg 0$).

\begin{rem}
If we make a change of variables via the `shear matrix'
\[
 \begin{pmatrix} v' \\ u' \end{pmatrix} = \begin{pmatrix} 1 & m-\tfrac{e}{2} \\ 0 & 1 \end{pmatrix} \begin{pmatrix} v \\ u \end{pmatrix}
\]
then the relation \eqref{eq:vhyperbolaequation} can be rewritten as
\[
  m + \alpha = u'v' + e
\]
while $\omega$ can be rewritten as $\omega = u'(\Theta + \tfrac{e}{2}f) + v'f$.  Then $Z_\omega (F)$ is a Laurent polynomial in $v'$, and $Z^l$-stability can equivalently be defined by letting $v' \to \infty$, in which case $Z^l$-stability is indeed a polynomial stability in the sense of Bayer.  Nonetheless, we will use the coordinates $(v,u)$ instead of $(v',u')$ in the rest of this article.
\end{rem}

\paragraph[Torsion triple and torsion quintuple in $\Bc^l$] We \label{para:Bltorsiontriplequintuple} now define the following subcategories of $\Tc^l, \Fc^l$
\begin{align*}
  \Tc^{l,+} &= \langle F \in \Coh^{=2}(X) : F \text{ is $\mu_f$-semistable}, \mu_f (F) > 0 \rangle, \\
  \Tc^{l,0} &= \{ F \in \Tc^l : F \text{ is $\mu_f$-semistable}, \mu_f (F)=0 \}, \\
  \Fc^{l,0} &= \{ F \in \Fc^l : F \text{ is $\mu_f$-semistable}, \mu_f (F) =0 \}, \\
  \Fc^{l,-} &= \langle F \in \Coh^{=2}(X) : F \text{ is $\mu_f$-semistable}, \mu_f (F) < 0 \rangle.
\end{align*}
For the same reason as in \cite[Remark 4.8(iii)]{Lo14}, we have the inclusion $\Tc^{l,0} \subset W_{1,\whPhi}$.  Since $W_{0,\whPhi} \subset \Tc^l$ from \ref{para:TlFlproperties}(iii), we have the torsion triple in $\Bl$
\begin{equation}\label{eq:Bltorsiontriple}
  (\Fc^l [1], \,\, W_{0,\whPhi},\,\, W_{1,\whPhi} \cap \Tc^l )
\end{equation}
which is an analogue of \cite[(4.12)]{Lo14}.  Also, by considering the $\mu_f$-HN filtrations of objects in $\Fc^l$ and $\Tc^l$, we obtain the torsion quintuple in $\Bc^l$
\begin{equation}\label{eq:Bltorsionquintuple}
  (\Fc^{l,0}[1],\,\, \Fc^{l,-}[1],\,\, \Coh^{\leq 1}(X),\,\, \Tc^{l,+},\,\, \Tc^{l,0} )
\end{equation}
which is an analogue of \cite[(4.13)]{Lo14}.

\paragraph[The category $W_{1,\whPhi} \cap \Tc^l$] From \label{para:W1xTcldecomposition} the torsion  quintuple \eqref{eq:Bltorsionquintuple}, we see that for every object $F \in W_{1,\whPhi} \cap \Tc^l$, the $\Tc^{l,+}$-component must be zero, or else  such a component would contribute a positive intersection number $f\ch_1$; this implies that  $F$ has a two-step filtration $F_0 \subseteq F_1 = F$ in $\Coh (X)$ where $F_0 \in W_{1,\whPhi} \cap \Coh^{\leq 1}(X)$ and is thus a $\whPhi$-WIT$_1$ fiber sheaf, while $F_1/F_0 \in \Tc^{l,0}$.  Since $f\ch_1$ is zero for both $F_0$ and $F_1/F_0$, the transform $\whPhi F[1]$ must be a torsion sheaf.

\paragraph[Transforms of torsion-free sheaves] The \label{para:torsionfreeshtransf}  torsion triple \eqref{eq:Bltorsiontriple} in $\Bl$ is taken by $\whPhi$ to the torsion  triple
\[
  (\wh{\Phi} \Fc^l [1], W_{1,\Phi}, \wh{\Phi} (W_{1,\whPhi} \cap \Tc^l) )
\]
in the abelian category $\wh{\Phi} \Bl$.  This implies that the heart $\wh{\Phi} \Bl [1]$ is a tilt of $\Coh (X)$ with respect to the torsion pair $(\mathcal T, \mathcal F)$ where
\begin{align*}
  \mathcal T &= \wh{\Phi} (W_{1,\whPhi} \cap \Tc^l) [1] \\
  \mathcal F &= \langle \wh{\Phi} \Fc^l [1], W_{1,\whPhi} \rangle.
\end{align*}
By \ref{para:W1xTcldecomposition}, we know $\mathcal T \subseteq \Coh^{\leq 1}(X)$.  Consequently, for every torsion-free sheaf $F$ on $X$ we have $F \in \mathcal F \subset \wh{\Phi} \Bl$, which implies ${\Phi} F [1]\in \Bl$.

\paragraph[Phases of  objects] We \label{para:phasesofobjects} analyse the phases of various objects in $\Bl$ with respect to $Z^l$-stability.   Note that if $F \in D^b(X)$ satisfies
\begin{align}
    \tilde{n} &= \ch_0(F), \notag\\
  \tilde{d} = f\ch_1(F)&, \text{\quad}  \tilde{c} = \Theta\ch_1(F), \notag\\
  \tilde{s} &= \ch_2(F) \label{eq:chFnotation}
\end{align}
then
\begin{align*}
  Z_\omega (F) &= -\ch_2(F) + \tfrac{\omega^2}{2}\ch_0(F) + i \omega \ch_1(F) \\
  &= -\tilde{s} + \left( (m-\tfrac{e}{2})u^2 + uv \right) \tilde{n} + i  ( u(\tilde{c} + m\tilde{d}) + v\tilde{d}) \\
  &= -\tilde{s} + (\alpha + (m-e))\tilde{n} + i ( u(\tilde{c} + m\tilde{d}) + v\tilde{d}) \text{ under the constraint \eqref{eq:vhyperbolaequation}}.
\end{align*}
Now further assume $F$ is a nonzero object of $\Bc^l$.  Consider the following scenarios:
\begin{itemize}
\item[(1)] $F \in \Coh^{\leq 0}(X)$.  Then $\ch_2(F)>0$, and so $Z_\omega (F) \in \mathbb{R}_{<0}$, giving $\phi (F)=1$.
\item[(2)] $F \in \Coh^{\leq 1}(X)$ and $\dimension F = 1$.  Then $\tilde{n}=0$.  We have $\tilde{d} =f\ch_1(F) \geq 0$ in this case.
    \begin{itemize}
    \item[(2.1)] If $\tilde{d}>0$, then $\phi (F) \to \tfrac{1}{2}$.
    \item[(2.2)] If $\tilde{d}=0$, then the effective divisor $\ch_1(F)$  is a positive multiple of the fiber class $f$, and so $(\Theta + mf)\ch_1(F)=\Theta \ch_1(F) = \tilde{c} > 0$, i.e.\ $\Im Z_\omega (F) = u\tilde{c}>0$.
        \begin{itemize}
        \item[(2.2.1)] If $\tilde{s} = \ch_2(F) > 0$ then $\phi (F) \to 1$.
        \item[(2.2.2)] If $\tilde{s} = 0$ then $\phi (F)=\tfrac{1}{2}$.
        \item[(2.2.3)] If $\tilde{s}<0$ then $\phi (F) \to 0$.
        \end{itemize}
    \end{itemize}
\item[(3)] $F \in \Coh^{=2}(X)$ and $f\ch_1(F)=\tilde{d}>0$.  Then $\phi (F) \to \tfrac{1}{2}$.
\item[(4)] $F \in \Tc^{l,0}$.  From the definition of $\Tc^{l,0}$, we have $\tilde{d}=f\ch_1(F)=0$ while $(\Theta + mf)\ch_1(F) >0$; we also know $F$ is $\wh{\Phi}$-WIT$_1$ from \ref{para:Bltorsiontriplequintuple}.  Thus $\wh{F}=\wh{\Phi} F [1]$ is a sheaf of rank zero, and so $\omega \ch_1 (\wh{F})$ must be strictly positive (if $\omega \ch_1(\wh{F})=0$, then $\wh{F}$ would be supported in dimension 0, implying $F$ itself is a fiber sheaf, a contradiction).  Thus from the discussion in \ref{para:ch1realZwcomparison} we know
    \[
    0 < -\Re Z_\omega ({\Phi} \wh{F} [1]) =  \Re Z_\omega (F)
    \]
    and hence $\phi (F) \to 0$.
\item[(5)] $F = A[1]$ where $A \in \Fc^{l,0}$.  Then $f\ch_1(A)=0$ and $(\Theta + mf)\ch_1(A) \leq 0$.  In this case, $A$ is $\wh{\Phi}$-WIT$_1$ by \ref{para:TlFlproperties}(v).  By a similar computation as in (4), we have
    \[
      0 < -\Re Z_\omega ({\Phi} \wh{A} [1])=-\Re Z_\omega (A[1]) = -\Re Z_\omega (F)
    \]
    and so $\phi (F) \to 1$.
\item[(6)] $F = A[1]$ where $A \in \Fc^{l,-}$.  Then $f\ch_1(A)<0$, i.e.\ $f\ch_1(F)>0$.  Hence $\phi (F) \to \tfrac{1}{2}$.
\end{itemize}


\section{Slope stability vs limit Bridgeland stability}\label{sec:maintheorem}

Given any torsion-free sheaf $E$ on $X$, we saw in  \ref{para:torsionfreeshtransf} that $\Phi E [1]$ lies in the heart $\Bl$.  In this section, we establish a comparison between   $\mu_\oo$-stability on $E$ and  $Z^l$-stability on the shifted transform $\Phi E [1]$ in the form of Theorem \ref{thm:Lo14Thm5-analogue}.  This theorem is the surface analogue of  \cite[Theorem 5.1]{Lo14}:

\begin{thm}\label{thm:Lo14Thm5-analogue}
Let $p : X \to B$ be a Weierstra{\ss} surface.
\begin{itemize}
\item[(A)] Suppose $E$ is a $\mu_\oo$-stable torsion-free sheaf on $X$ and $B = \tfrac{e}{2}f$.
\begin{itemize}
\item[(A1)] If $\oo \ch_1^B(E)>0$, then $\Phi E[1]$ is a $Z^l$-stable object in $\Bl$.
\item[(A2)] If $\oo \ch_1^B(E)=0$, then $\Phi E[1]$ is a $Z^l$-semistable object in $\Bl$, and the only $\Bl$-subobjects $G$ of $\Phi E[1]$ where $\phi (G)=\phi (\Phi E [1])$ are objects in  $\Phi (\Coh^{\leq 0}(X))$.
\item[(A3)] If $E$ is locally free, then $\Phi E[1]$ is a $Z^l$-stable object in $\Bl$.
\end{itemize}
\item[(B)] Suppose $F \in \Bl$ is a $Z^l$-semistable object with $f\ch_1(F) \neq 0$, and $F$ fits in the $\Bl$-short exact sequence (which exists by  \eqref{eq:Bltorsiontriple})
\[
 0 \to F' \to F \to F'' \to 0
\]
where $F' \in \langle \Fc^l[1], W_{0,\whPhi} \rangle$ and $F'' \in \langle W_{1,\whPhi} \cap \Tc^l\rangle$.  Then $\whPhi F'$ is a $\mu_\oo$-semistable torsion-free sheaf on $X$.
\end{itemize}
\end{thm}

Even though the proof of Theorem \ref{thm:Lo14Thm5-analogue} is analogous to that of \cite[Theorem 5.1(A)]{Lo14}, we include most of the details for ease of reference, and also to lay out explicitly the necessary changes to the proof of \cite[Theorem 5.1]{Lo14}.

\begin{proof}[Proof of Theorem \ref{thm:Lo14Thm5-analogue}(A)]
Let us write $F = \Phi E [1]$  throughout the proof.  Since $\rank (E) \neq 0$, we have $\phi (F) \to \tfrac{1}{2}$.  Take any $\Bl$-short exact sequence
\begin{equation}
0 \to G \to F \to F/G \to 0
\end{equation}
where $G \neq 0$.  This yields a long exact sequence of sheaves
\begin{equation}\label{eq:lem:Lo14Thm5-1A-analogue-eq2}
0 \to \wh{\Phi}^0 G \to E \overset{\alpha}{\to} \wh{\Phi}^0 (F/G) \to \wh{\Phi}^1 G \to 0
\end{equation}
and we see $\whPhi^1 (F/G)=0$.  From the torsion triple \eqref{eq:Bltorsiontriple} in $\Bl$, we know $G$ fits in  the exact triangle
\[
  \Phi (\whPhi^0 G)[1] \to G \to \Phi (\whPhi^1 G) \to \Phi (\whPhi^0 G)[2]
\]
where $\Phi (\whPhi^0 G)[1] \in \langle \Fc^l [1], W_{0,\whPhi}\rangle$ is precisely the $\whPhi$-WIT$_0$ component of $G$, and $\Phi (\whPhi^1 G) \in W_{1,\whPhi} \cap \Tc^l$ the $\whPhi$-WIT$_1$ component of $G$.

\textbf{Suppose $\rank (\image \alpha) = 0$.} Then $\rank (\whPhi^0 G) = \rank E > 0$, and so $f\ch_1 (\Phi (\whPhi^0 G)[1])>0$.  Now we break into two cases:
\begin{itemize}
\item[(a)] $\ch_1 (\image \alpha) \neq 0$.  Then $\mu_{\oo,B} (\whPhi^0 G) < \mu_{\oo,B} (E)$, which implies $\phi (\Phi (\whPhi^0 G)[1]) \prec \phi (F)$.
    \begin{itemize}
    \item[(i)] If $\dimension \Phi (\whPhi^1 G)=2$:  from \ref{para:W1xTcldecomposition} we know $\Phi (\whPhi^1 G)$ fits in a short exact sequence of sheaves
        \begin{equation}\label{eq:PhiwhPhi1G-decomp}
          0 \to A' \to \Phi (\whPhi^1 G) \to A'' \to 0
        \end{equation}
        where $A' \in W_{1,X} \cap \Coh^{\leq 1}(X) \subset \Coh (\pi)_0$ and $A'' \in \Tc^{l,0}$.  Thus $f\ch_1(\Phi (\whPhi^1 G))=0$, and $Z_\omega (F)$ is dominated by its real part.  From the computation in \ref{para:ch1realZwcomparison}, we know $\Re Z_\omega (\Phi (\whPhi^1 G)) > 0$, and so $\phi (\Phi (\whPhi^1 G)) \to 0$, giving us $\phi (G) \prec \phi (F)$ overall.

    \item[(ii)] If $\dimension \Phi (\whPhi^1 G) \leq 1$: then the component $A''$ in (i) vanishes, and $\Phi (\whPhi^1 G) = A'$ is a $\whPhi$-WIT$_1$ fiber sheaf.  Then
        \[
          Z_\omega (\Phi (\whPhi^1 G))= -\bar{s} + i\bar{c}u
        \]
        where $\bar{s} = \ch_2(A') \leq 0$ while $\bar{c} = \Theta\ch_1(A')\geq 0$.

        If $\bar{s}<0$, then again we  have $\phi (G) \prec \phi (F)$.  On the other hand, if $\bar{s}=0$ then the order of magnitude of $Z_\omega (\Phi (\whPhi^1 G))$ as $v \to \infty$ is  $O(\tfrac{1}{v})$, and so we still have $\phi (G) \prec \phi (F)$ overall.
    \end{itemize}
\item[(b)] $\ch_1 (\image \alpha) =0$.  Then $\image \alpha \in \Coh^{\leq 0}(X)$, in which case $\ch_i (\whPhi^0 G) = \ch_i (E)$ for $i=0,1$.  From the cohomological Fourier-Mukai transform \eqref{eq:cohomFMT-Phi}, it follows that $\ch_0, f\ch_1$ and $\ch_2$ of
 $\Phi (\whPhi^0 G)[1]$ and $F$ agree; from \eqref{eq:ZwPhiEshift} we also see that  all the terms of  $Z_\omega ( \Phi (\whPhi^0 G)[1])$ and $Z_\omega(F)$ agree except the terms involving $u$.  As in (a)(i), we have a decomposition of $\Phi (\whPhi^1 G)$ of the form \eqref{eq:PhiwhPhi1G-decomp}.
    \begin{itemize}
    \item[(i)] If $\dimension \Phi (\whPhi^1 G)=2$: then $A'' \neq 0$, and we have $\Re Z_\omega (A'') > 0$ by \ref{para:ch1realZwcomparison} while $\Im Z_\omega (A'')$ has order of magnitude $O(\tfrac{1}{v})$.  On the other hand, $A'$ is a $\whPhi$-WIT$_1$ fiber sheaf and so $\Re Z_\omega (A') \geq 0$ while $\Im Z_\omega (A')$ also has order of magnitude $O(\tfrac{1}{v})$.  Overall, we have $\phi (G) \prec \phi (F)$.
    \item[(ii)] If $\dimension \Phi (\whPhi^1 G)\leq 1$: then $A''=0$ and $\Phi (\whPhi^1 G)=A'$ is a $\whPhi$-WIT$_1$ fiber sheaf with $\ch_2 (A') \leq 0$. With $\bar{s}, \bar{c}$ as in (a)(ii) above, we observe:
    \begin{itemize}
    \item If $\bar{s}<0$, then $\Re Z_\omega (\Phi (\whPhi^1 G))>0$ while $\Im Z_\omega (\Phi (\whPhi^1 G))$ has magnitude $O(\tfrac{1}{v})$, giving us  $\phi (G) \prec \phi (F)$ overall.
    \item If $\bar{s}=0$, then $\bar{c} \geq 0$ (with $\bar{c}=0$ iff $A'=0$) and $\whPhi^1 G \in \Coh^{\leq 0}(X)$.  Thus $\whPhi^0 (F/G)$ also lies in $\Coh^{\leq 0}(X)$ from the exact sequence \eqref{eq:lem:Lo14Thm5-1A-analogue-eq2}.  Since $F/G \in \Bl$, from the torsion triple \eqref{eq:Bltorsiontriple} in $\Bl$ we know $\whPhi^0 (F/G) \in \langle \whPhi \Fc^l [1], W_{1,\whPhi} \rangle$, i.e.\ $\whPhi^0 (F/G)$ is the extension of a sheaf in $W_{1,\whPhi}$ by a sheaf in $\whPhi \Fc^l [1]$.  However, every nonzero coherent sheaf in $\whPhi \Fc^l [1]$ has $f\ch_1 \neq 0$, and so must be supported in dimension at least 1.  Thus the $\whPhi \Fc^l [1]$-component of $\whPhi^0 (F/G)$ must  vanish, i.e.\ $\whPhi^0 (F/G)$ lies in $W_{1,\whPhi} \cap \Coh^{\leq 0}(X)$, which forces $\whPhi^0 (F/G)$ to be zero.  Then $F/G$ itself is zero, i.e.\ $G=F$.
    \end{itemize}
    \end{itemize}
\end{itemize}

\textbf{Suppose $\rank (\image \alpha) > 0$.} If $\whPhi^0 G \neq 0$ then $0 < \rank (\whPhi^0 G) < \rank (E)$ and so $\mu_{\oo,B} (\whPhi^0 G) < \mu_{\oo,B} (E)$, and so same argument as in part (a) above shows that $\phi (G) \prec \phi (F)$.  From now on, let us assume $\whPhi^0 G =0$, in which case we have the exact sequence of sheaves
\[
  0 \to E \to \whPhi^0 (F/G) \to \whPhi^1 G \to 0.
\]
Thus $G$ is a $\whPhi$-WIT$_1$ object, and from the torsion triple \eqref{eq:Bltorsiontriple} in $\Bl$ we see that $G$ must lie in $W_{1,\whPhi} \cap \Tc^l$.  As in case (a)(i) above, $G$ fits in a short exact sequence in $\Coh (X)$
\[
  0 \to A' \to G \to A'' \to 0
\]
where $A'$ is a $\whPhi$-WIT$_1$ fiber sheaf and $A'' \in \Tc^{l,0}$. We now divide into the following cases:
\begin{itemize}
\item $A'' \neq 0$: then we know $\Re Z_\omega (A'')$ is positive from \ref{para:ch1realZwcomparison} and is $O(1)$, while $\Im Z_\omega (A'')$ is $O(\tfrac{1}{v})$.  On the other hand, since $\ch_2(A') \leq 0$ we know $\Re Z_\omega (A')$ is nonnegative and $O(1)$,  while $\Im Z_\omega (A')$ is  $O(\tfrac{1}{v})$.  Overall, we have $\phi (G) \to 0$, giving us $\phi (G) \prec \phi (F)$.
\item $A''=0$ and $\ch_2 (A')<0$:  then $\phi (G) \to 0$  and we still have $\phi (G) \prec \phi (F)$.
\item $A''=0$ and $\ch_2 (A') =0$:  in this case $A' \in \Phi (\Coh^{\leq 0}(X))$ and so $\phi (G) = \tfrac{1}{2}$.  This is the most intricate of all the cases in this proof to treat, and we single out the following two  scenarios:
    \begin{itemize}
    \item[(S1)] If $\oo \ch_1^B(E)>0$: then $\Re Z_\omega (E)<0$ by \eqref{eq:ooch1twistReZePhiE1relation}, which gives $\phi (F) \succ \tfrac{1}{2} = \phi (G)$.  (Note that this is despite  $\phi (F) \to \tfrac{1}{2}$.)  Therefore, if $\oo \ch_1^B(E)>0$ then $\Phi E[1]$ is always $Z^l$-stable.  This proves statement (A1).
    \item[(S2)] If $\oo \ch_1^B(E)=0$: then $\Re Z_\omega (F)=0$, and $\phi (F) =\tfrac{1}{2} = \phi (G)$.  In this case, $\Phi E[1]$ is $Z^l$-semistable, and it would be strictly $Z^l$-semistable if and only if there exists a $\Bl$-subobject $G$ of $\Phi E[1]$ as in this case. This proves statement (A2).
    \end{itemize}

    Of course,   scenarios (S1) and (S2) above can be ruled out if we impose the vanishing $\Hom (\Phi (\Coh^{\leq 0}(X)),F)=0$, i.e.\ $\Hom (\Phi Q,F)=0$ for every $Q \in \Coh^{\leq 0}(X)$.  Note that for any $Q \in \Coh^{\leq 0}(X)$,
    \[
    \Hom (\Phi Q,F)=\Hom (Q, \whPhi F [1])=\Hom (Q,E[1])=\Ext^1 (Q,E).
    \]
    Hence $\Hom (\Phi (\Coh^{\leq 0}(X)),F)=0$ if and only if $\Ext^1 (Q,E)=0$ for every $Q \in \Coh^{\leq 0}(X)$, which in turn is equivalent to $E$ being a locally free sheaf by Lemma \ref{lem:surfaceshlocallyfreechar} below.  This proves statement (A3), and completes the proof of part (A).
\end{itemize}
\end{proof}

\begin{lem}\label{lem:surfaceshlocallyfreechar}
Suppose $E$ is a torsion-free sheaf $E$ on a smooth projective surface $X$.  Then $E$ is locally free if and only if $\Ext^1 (T,E)=0$ for every $T \in \Coh^{\leq 0}(X)$.
\end{lem}

\begin{proof}
Consider the short exact sequence of sheaves
\[
0 \to E \to E^{\ast \ast} \to Q \to 0
\]
where $Q$ is necessarily a sheaf in $\Coh^{\leq 0}(X)$.  If $E$ is not locally free, then $Q \neq 0$ and we have  $\Ext^1 (Q,E) \neq 0$.  On the other hand, if $E$ is locally free then for any $T \in \Coh^{\leq 0}(X)$ we have $\Ext^1 (T,E)\cong \Ext^1 (E,T\otimes \omega_X) \cong H^1 (X,E^\ast \otimes T)=0$.
\end{proof}

\begin{proof}[Proof of Theorem \ref{thm:Lo14Thm5-analogue}(B)]
Let $F', F, F''$ be as in the statement of the theorem.  We begin by showing that $\whPhi F'$ is a torsion-free sheaf, i.e.\ $\Hom (\Coh^{\leq 1}(X), \whPhi F')=0$, i.e.\
\begin{equation}\label{lem:Lo14Thm5-1B-analogue-eq1}
\Hom (\Phi \Coh^{\leq 1}(X)[1], F')=0.
\end{equation}

Proceeding as in the proof of \cite[Lemma 5.8]{Lo14}, we observe
\begin{align*}
  \Phi \Coh^{\leq 1}(X)[1]  &\subset \langle \{ E \in W_{1,\whPhi} : f\ch_1(E)=0\}, \scalea{\gyoung(;;+,;;+)}[-1], \Coh^{\leq 0}(X)[-1]\rangle [1] \\
  &\subset \langle \Coh (X)[1], \scalea{\gyoung(;;+,;;+)}, \Coh^{\leq 0}(X) \rangle \\
  &\subset \langle \Bl[1], \Bl\rangle.
\end{align*}
Therefore, in order to prove the vanishing \eqref{lem:Lo14Thm5-1B-analogue-eq1}, it suffices to show the following two things:
\begin{itemize}
\item[(i)] For any $G \in W_{1,\whPhi}$ with $f\ch_1(G)=0$, we have $\Hom_{\Bl} (\mathcal{H}^0_{\Bl} (G[1]),F')=0$.
\item[(ii)] $\Hom (\langle\, \scalea{\gyoung(;;+,;;+)}, \Coh^{\leq 0}(X) \rangle, F' )=0$.
\end{itemize}

For (i), let us consider the $(\Tc^l,\Fc^l)$-decomposition of $G$ in $\Coh (X)$
\[
0 \to G' \to G \to G'' \to 0.
\]
This shows $\mathcal{H}^0_{\Bl}(G[1]) = G''[1]$.  Since $G$ is a $\whPhi$-WIT$_1$ sheaf, so is its subsheaf $G'$; thus $G' \in W_{1,\whPhi} \cap \Tc^l$, and  from \ref{para:W1xTcldecomposition} we have  $f\ch_1(G')=0$.  Since $f\ch_1(G)=0$, we also have $f\ch_1(G'')=0$.  By considering the $\mu_f$-HN filtration of $G''$, we obtain  $G'' \in \Fc^{l,0}$.

For any $\Bl$-morphism $\alpha : G''[1] \to F'$ and with $\Ac_1$ defined as in \eqref{eq:A1definition} below, we now have $\image \alpha \in \Ac_1$  and $\phi (\image \alpha) \to 1$ by Lemma \ref{lem:Lo14Lem5-7analogue} below.  However, this gives a composition of $\Bl$-injections
\[
\image \alpha \hookrightarrow F' \hookrightarrow F.
\]
  Hence $\alpha$ must be zero, or else $F$ would be destabilised, proving (i).  A similar argument as above proves (ii).  Hence $\whPhi F'$ is a torsion-free sheaf on $X$.

Next, we show that $\whPhi F'$ is $\mu_\oo$-semistable.  Take any short exact sequence of coherent sheaves on $X$
\[
  0 \to B \to \whPhi F' \to C \to 0
\]
where $B, C$ are both torsion-free sheaves.  Then $\Phi [1]$ takes this short exact sequence to a $\Bl$-short exact sequence
\[
0 \to \Phi B[1] \to F' \to \Phi C [1] \to 0
\]
by \ref{para:torsionfreeshtransf}.  The $Z^l$-semistability of $F$ gives $\phi (\Phi B[1]) \preceq \phi (F)$, which implies $\mu_\oo (B) \leq \mu_\oo (\whPhi F)$.  On the other hand, since $F''$ is precisely the $\whPhi$-WIT$_1$ component of $H^0(F)$, by Lemma \ref{lem:Lo14Lem5-9analogue} below we have $F'' \in \Phi \Coh^{\leq 0}(X)$, i.e.\ $\whPhi F'' \in \Coh^{\leq 0}(X)[-1]$.  This gives
\[
  \mu_\oo (\whPhi F') = \mu_\oo (\whPhi F) \geq \mu_\oo (B).
\]
Hence $\whPhi F'$ is a $\mu_\oo$-semistable torsion-free sheaf.
\end{proof}

Let us define
\begin{equation}\label{eq:A1definition}
  \Ac_1 = \langle \Coh^{\leq 0}(X), \scalea{\gyoung(;;+,;;+)}, \, \Fc^{l,0}[1]\rangle.
\end{equation}

\begin{lem}\label{lem:Lo14Lem5-7analogue}
The category $\Ac_1$ is closed under quotient in $\Bl$, and every object in this category satisfies $\phi \to 1$.
\end{lem}

\begin{proof}
The second part of the lemma follows from the computations in \ref{para:phasesofobjects}.  For the first part, take any $A \in \Ac_1$ and consider any $\Bl$-short exact sequence of the form
\[
0 \to A' \to A \to A'' \to 0.
\]
We need to show that $A'' \in \Ac_1$.  Recall that $\Bl_0 = \{ F \in \Bl : f\ch_1(F)=0\}$ is a Serre subcategory of $\Bl$; also note that $\Ac_1$ is contained in $\Bl_0$.  Hence $A''$ lies in $\Bl_0$, meaning $H^{-1}(A'') \in \Fc^{l,0}[1]$.  On the other hand, since $H^0(A)\in \langle \Coh^{\leq 0}(X), \scalea{\gyoung(;;+,;;+)} \,\rangle$ from the definition of $\Ac_1$, we also have $H^0(A'') \in \langle \Coh^{\leq 0}(X), \scalea{\gyoung(;;+,;;+)}\,\rangle$. Thus $A'' \in \Ac_1$, and we are done.
\end{proof}

\begin{lem}\label{lem:Lo14Lem5-9analogue}
Suppose $F \in \Bl$ is a $Z^l$-semistable object with $f\ch_1(F) \neq 0$.  Then the $\whPhi$-WIT$_1$ component of $H^0(F)$ lies in $\Phi \Coh^{\leq 0}(X)$.
\end{lem}

\begin{proof}
Let $G$ denote the $\whPhi$-WIT$_1$ component of $H^0(F)$.  With respect to the torsion triple \eqref{eq:Bltorsiontriple} in $\Bl$, this is precisely the $W_{1,\whPhi} \cap \Tc^l$ component of $F$.  Hence by \ref{para:W1xTcldecomposition}, $G$ has a two-step filtration $G_0 \subseteq G_1 = G$ in $\Coh (X)$ such that $G_1/G_0 \in \Tc^{l,0}$ and $G_0$ is a $\whPhi$-WIT$_1$ fiber sheaf (and so $\ch_2(G_0) \leq 0$).  Now we have a composition of $\Bl$-surjections
\[
  F \twoheadrightarrow G \twoheadrightarrow G_1/G_0
\]
with $\phi (F) \to \tfrac{1}{2}$ while $\phi (G_1/G_0) \to 0$ from \ref{para:phasesofobjects}(4).  Since $F$ is assumed to be $Z^l$-semistable, this forces $G_1/G_0=0$, and so $G=G_0$.

Suppose now that $\bar{c} = \Theta \ch_1(G)$ and $\bar{s} = \ch_2(G)$.  Then
\[
  Z_\omega (G) = -\bar{s} + i\bar{c}u.
\]
By the $Z^l$-semistability of $F$, the fiber sheaf $G$ cannot have any quotient sheaf with $\ch_2<0$ (such a quotient would have $\phi \to 0$ by \ref{para:phasesofobjects}(2.2.3), destabilising $F$).  Hence $G$ is a slope semistable fiber sheaf with $\ch_2=0$, implying $G \in \Phi \Coh^{\leq 0}(X)$ \cite[Proposition 6.38]{FMNT}.
\end{proof}

\section{The Harder-Narasimhan property of limit Bridgeland stability}\label{sec:HNproperty}

Following the line of thought in \cite[Section 6]{Lo14}, we begin by constructing a torsion triple in $\Bl$ that separates objects of distinct phases.  Recall the definition \eqref{eq:A1definition}
\[
  \Ac_1 = \langle \Coh^{\leq 0}(X), \scalea{\gyoung(;;+,;;+)}, \, \Fc^{l,0}[1]\rangle.
\]

\begin{lem}\label{lem:Lo14Lem6-1analogue}
The category $\Ac_1$ is a torsion class in $\Bl$.
\end{lem}

\begin{proof}
We already showed in Lemma \ref{lem:Lo14Lem5-7analogue} that $\Ac_1$ is closed under quotient in $\Bl$.  It remains to show that every object $F \in\Bl$ is the extension of an object in $\Ac_1^\circ$ by an object in $\Ac_1$.

For any $F \in \Bl$, consider the $\Bl$-short exact sequence
\[
0 \to G[1] \to F \to F' \to 0
\]
where $G[1]$ is the $\Fc^{l,0}[1]$-component of $F$ with respect to the torsion quintuple \ref{eq:Bltorsionquintuple}; equivalently, $G$ is the $\Fc^{l,0}$-component of $H^{-1}(F)$.  Note that $\Hom (\Fc^{l,0}[1],F')=0$ by construction.

Suppose $F' \notin \Ac_1^\circ$.  Then there exists a nonzero morphism $\beta : U \to F'$ where $U \in \Ac_1$.  Since $\Ac_1$ is closed under quotient in $\Bl$, we can replace $U$ by $\image \beta$ and assume $\beta$ is a $\Bl$-injection.  The vanishing  $\Hom (\Fc^{l,0}[1],F')=0$ then implies $H^{-1}(U)=0$ and so $U = H^0(U) \in \langle \Coh^{\leq 0}(X), \scalea{\gyoung(;;+,;;+)} \,\rangle$.

Suppose we have an ascending chain in $\Bl$
\[
U_1 \subseteq U_2 \subseteq \cdots \subseteq U_m \subseteq \cdots \subseteq F'
\]
where $U_i \in \langle \Coh^{\leq 0}(X), \scalea{\gyoung(;;+,;;+)}  \,\rangle$ for all $i$.  This induces an ascending chain of coherent sheaves
\[
  \whPhi^0 U_1 \subseteq \whPhi^0 U_2 \subseteq \cdots \subseteq \whPhi^0 F'.
\]
Thus the $U_i$ must stabilise, i.e.\ there exists a maximal $\Bl$-subobject $U$ of $F'$ lying in $\langle \Coh^{\leq 0}(X), \scalea{\gyoung(;;+,;;+)}  \,\rangle$.  Applying the octahedral axiom to the $\Bl$-surjections $F \twoheadrightarrow F' \twoheadrightarrow F'/U$ gives the diagram
\[
\scalebox{0.8}{
\xymatrix{
  & & & G[2] \ar[ddd] \\
  & & & \\
  & F' \ar[dr] \ar[uurr] & & \\
  F \ar[ur] \ar[rr] & & F'/U \ar[r] \ar[dr] & M [1] \ar[d] \\
  & & & U [1]
}
}
\]
in which every straight line is an exact triangle, and for some $M \in \Bl$.  The vertical exact triangle gives $H^{-1}(M) \cong G$ and $H^0(M) \cong U$, and so $M \in \Ac_1$.  A similar argument as in the proof of \cite[Lemma 6.1(b)]{Lo14} then shows that $F'/U \in \Ac_1^\circ$, thus  finishing the proof.
\end{proof}

We now define
\begin{align}
\Ac_{1,1/2} &:= \langle \Ac_1, \Fc^{l,-}[1], \scalea{\gyoung(;;+,;;0)},  \scalea{\gyoung(;;*,;+;*)}, \scalea{\gyoung(;+;*,;+;*)} \, \rangle \notag \\
&= \langle \Fc^l[1], \xymatrix @-2.3pc{
\scalea{\gyoung(;;,;;+)} &  \scalea{\gyoung(;;+,;;+)} &  \scalea{\gyoung(;;*,;+;*)} & \scalea{\gyoung(;+;*,;+;*)} \\
& \scalea{\gyoung(;;+,;;0)} & &
}
\rangle.   \label{eq:A1and1over2definition}
\end{align}

\begin{lem}\label{lem:Lo14Lem6-2analogue}
$\Ac_{1,1/2}$ is a torsion class in $\Bl$.
\end{lem}

\begin{proof}
For the purpose of this proof, let us write
\[
  \mathcal{E} =  \xymatrix @-2.3pc{
\scalea{\gyoung(;;,;;+)} &  \scalea{\gyoung(;;+,;;+)} &  \scalea{\gyoung(;;*,;+;*)} & \scalea{\gyoung(;+;*,;+;*)} \\
& \scalea{\gyoung(;;+,;;0)} & &
}.
\]
(Recall that concatenation of 2 by 2 boxes of the form $\scalea{\gyoung(;;,;;)}$ means their extension closure.) It is easy to check that $\mathcal{E}$ is a torsion class in $\Coh (X)$ and that
\[
  \mathcal{E} = \{ H^0(F) : F \in \Ac_{1,1/2}\}.
\]
The same argument as in \cite[Lemma 6.2]{Lo14} then shows that every object in $\Bc^l$ can be written as the extension of an object in $\mathcal{E}$ by an object in $\Ac_{1,1/2}$, proving the lemma.
\end{proof}

Now that we know $\Ac_1, \Ac_{1,1/2}$ are both torsion classes in $\Bl$ with the inclusion $\Ac_1 \subseteq \Ac_{1,1/2}$, we can construct the torsion triple in $\Bl$
\begin{equation}\label{eq:torsiontripleAlowerstar}
  (\Ac_1, \,\Ac_{1,1/2} \cap \Ac_1^\circ, \,\Ac_{1,1/2}^\circ ).
\end{equation}
We have the following finiteness properties for the components of this torsion triple:

\begin{prop}\label{prop:Alowerstarfinitenessproperties}
The following finiteness properties hold:
\begin{itemize}
\item[(1)] For $\Ac = \Ac_1$:
\begin{itemize}
\item[(a)] There is no infinite sequence of strict monomorphisms in $\Ac$
\begin{equation}\label{eq:prop-finiteness-eq1}
  \cdots \hookrightarrow E_n \hookrightarrow \cdots \hookrightarrow E_1 \hookrightarrow E_0.
\end{equation}
\item[(b)] There is no infinite sequence of strict epimorphisms in $\Ac$
\begin{equation}\label{eq:prop-finiteness-eq2}
  E_0 \twoheadrightarrow E_1 \twoheadrightarrow \cdots \twoheadrightarrow E_n \twoheadrightarrow \cdots.
\end{equation}
\end{itemize}
\item[(2)] For $\Ac = \Ac_{1,1/2} \cap \Ac_1^\circ$:
    \begin{itemize}
    \item[(a)] There is no infinite sequence of strict monomorphisms \eqref{eq:prop-finiteness-eq1} in $\Ac$.
    \item[(b)] There is no infinite sequence of strict epimorphisms \eqref{eq:prop-finiteness-eq2} in $\Ac$.
    \end{itemize}
\item[(3)] For $\Ac = \Ac_{1,1/2}^\circ$:
  \begin{itemize}
  \item[(a)] There is no infinite sequence of strict monomorphisms \eqref{eq:prop-finiteness-eq1} in $\Ac$.
  \item[(b)] There is no infinite sequence of strict epimorphisms \eqref{eq:prop-finiteness-eq2} in $\Ac$.
  \end{itemize}
\end{itemize}
\end{prop}

Even though the proof of this proposition is modelled after that of \cite[Proposition 5.3]{Lo14}, we  lay out the details for clarity and ease of reference.  For instance, since the total space of our elliptic surface $X$ does not necessarily have Picard rank 2 as in \cite{Lo14}, the strategy of using the positivity of certain intersection numbers needs to be adjusted carefully.


\begin{proof}
In proving (1)(a), (2)(a) and (3)(a), we will consider the $\Bl$-short exact sequences
\begin{equation}\label{eq:prop-finiteness-eq3}
  0 \to E_{i+1} \overset{\beta_i}{\to} E_i \to G_i \to 0.
\end{equation}
On the other hand, in proving (1)(b), (2)(b) and (3)(b), we will consider the $\Bl$-short exact sequences
\begin{equation}\label{eq:prop-finiteness-eq4}
  0 \to K_i \to E_i \to E_{i+1} \to 0.
\end{equation}
Since $f\ch_1 \geq 0$ on $\Bl$ from \ref{para:TlFlproperties}(iv), we know $f\ch_1 (E_i)$ is a decreasing sequence of nonnegative integers when proving any of the six cases of this proposition.  Therefore, by omitting a finite number of terms in the sequence $E_i$ if necessary,  we can always assume that the  $f\ch_1 (E_i)$ are constant.  This also implies that $f\ch_1(G_i)=0$ and $f\ch_1(K_i)=0$ for all $i$, which in turn implies $f\ch_1(H^j(G_i))=0$ and $f\ch_1(H^j(K_i))=0$ for all $i,j$.

Throughout the proof, we will also fix an $m>0$ such that $\Theta + mf$ is an ample divisor on $X$.
\smallskip

 \textbf{(1)(a)}:  For any object $A \in \Fc^{l,0}[1]$, we know $f\ch_1(A)=0$ and $(\Theta + mf)\ch_1(A)=\Theta \ch_1(A) \geq 0$ by the definition of $\Fc^{l,0}$.  In addition,  any $A \in \langle \Coh^{\leq 0}(X), \scalea{\gyoung(;;+,;;+)}\,\rangle$ is a fiber sheaf and satisfies $\Theta \ch_1(A) \geq 0$.  Thus $\Theta \ch_1 \geq 0$ on $\Ac_1$, and by omitting a finite number of terms if necessary, we can assume that $\Theta \ch_1 (E_i)$ is constant and $\Theta \ch_1 (G_i)=0$ for all $i$.  Similarly, we can assume that $\ch_0 (E_i)$ is constant and $\ch_0 (G_i)=0$ for all $i$.

 That $\ch_0(G_i)=0$ implies $G_i = H^0(G_i)$, and so $G_i$ is a fiber sheaf.  That $\Theta \ch_1 (G_i)=0$ then implies $G_i$ must be supported in dimension 0.

 The long exact sequence of cohomology from \eqref{eq:prop-finiteness-eq3} now looks like
  \[
  0 \to H^{-1}(E_{i+1}) \to  H^{-1}(E_i)\to 0 \to H^0(E_{i+1}) \to H^0(E_i) \to H^0(G_i) \to 0,
  \]
  from which we see the $H^{-1}(E_i)$ stabilise.
  From the definition of $\Ac_1$, we also know $\ch_2 (H^0(A))\geq 0$ for  any $A \in \Ac_1$.  Thus $\ch_2 (H^0 (E_i))$ eventually stabilises, forcing $\ch_2 (H^0(G_i))=0$, in which case $G_i=H^0(G_i)=0$, i.e.\ the sequence  $E_i$ itself stabilises.

\textbf{(1)(b)}: from the long exact sequence of cohomology of \eqref{eq:prop-finiteness-eq4}, the $H^0(E_i)$ must eventually stabilise since $\Coh (X)$ is noetherian.  Hence let us suppose the $H^0(E_i)$ are constant.  The remainder of the long exact sequence reads
\[
0 \to H^{-1}(K_i) \to H^{-1}(E_i) \to H^{-1}(E_{i+1}) \to H^0(K_i) \to 0.
\]
Since $\ch_0\leq 0$ on $\Ac_1$, the sequence $\ch_0 (H^{-1}(E_i))$ eventually stabilises, so we can assume that $\ch_0 (H^{-1}(K_i))=0$ for all $i$ (noting $\ch_0(H^0(K_i))=0$), i.e.\ $H^{-1}(K_i)=0$, i.e.\ $K_i = H^0(K_i)$ is a fiber sheaf for all $i$.

As in (1)(a), we know $\Theta \ch_1 \geq 0$ on $\Ac_1$.  Hence $\Theta \ch_1 (E_i)$ eventually stabilises, giving $\Theta \ch_1 (K_i)=0$; since $K_i$ is a fiber sheaf, this forces $K_i$ to be supported in dimension 0.  The exact sequence above then gives
\[
  H^{-1}(E_i) \hookrightarrow H^{-1}(E_{i+1}) \hookrightarrow H^{-1}(E_{i+1})^{\ast\ast}
\]
where $H^{-1}(E_{i+1})^{\ast \ast}$ is independent of $i$ for $i \gg 0$ since $H^0(K_i) \in \Coh^{\leq 0}(X)$.  Thus the $H^{-1}(E_i)$ also stabilise, and the $E_i$ themselves stabilise.

\textbf{(2)(a)}: Recall from \eqref{eq:A1and1over2definition} that
\begin{equation*}
\Ac_{1,1/2} = \langle \Ac_1, \Fc^{l,-}[1], \scalea{\gyoung(;;+,;;0)},  \scalea{\gyoung(;;*,;+;*)}, \scalea{\gyoung(;+;*,;+;*)} \, \rangle = \langle \Fc^l[1], \xymatrix @-2.3pc{
\scalea{\gyoung(;;,;;+)} &  \scalea{\gyoung(;;+,;;+)} &  \scalea{\gyoung(;;*,;+;*)} & \scalea{\gyoung(;+;*,;+;*)} \\
& \scalea{\gyoung(;;+,;;0)} & &
}
\rangle.
\end{equation*}
Since  we can assume $f\ch_1(H^{-1}(G_i))=0$ and $f\ch_1(H^0(G_i))=0$, we have $H^{-1}(G_i) \in \Fc^{l,0}$ and know that $H^0(G_i)$ cannot have any subfactors in $\scalea{\gyoung(;;*,;+;*)}$ or $\scalea{\gyoung(;+;*,;+;*)}$.  Since $\beta_i$ is a strict morphism in $\Ac$, we have $G_i\in \Ac$ and so $\Hom (\Fc^{l,0}[1],G_i)=0$, i.e.\ $H^{-1}(G_i)=0$.  This leaves $G_i \in \langle \scalea{\gyoung(;;,;;+)}, \scalea{\gyoung(;;+,;;+)}, \scalea{\gyoung(;;+,;;0)}\rangle$, which means that $G_i$ is a fiber sheaf where all the HN factors with respect to the slope function $\ch_2/D\ch_1$ (for any ample divisor $D$ on $X$) have $\ch_2 \geq 0$.  Again by $\Hom (\Ac_1, G_i)=0$, we have $G_i \in \scalea{\gyoung(;;+,;;0)}$.

From the long exact sequence of cohomology of \eqref{eq:prop-finiteness-eq3}, we know the $H^{-1}(E_i)$ are constant and
\[
0 \to H^0(E_{i+1}) \to H^0 (E_i) \to H^0(G_i)\to 0
\]
is exact.  Applying the Fourier-Mukai functor $\whPhi$, we obtain the long exact sequence of sheaves
\[
0 \to \whPhi^0 (H^0(E_{i+1})) \to \whPhi^0 (H^0 (E_i)) \to 0 \to \whPhi^1 (H^0(E_{i+1})) \to \whPhi^1 (H^0(E_i)) \to \whPhi^1 (H^0(G_i)) \to 0.
\]
According to Lemma \ref{lem:Lo14Lem6-4analogue} below, $\whPhi^1 (H^0(E_i)) \in \Coh^{\leq 0}(X)$ for all $i$.  Hence $\whPhi^0(H^0(E_i)), \whPhi^1(H^0(E_i))$ both stabilise for  $i \gg 0$, i.e.\ $H^0(E_i)$ themselves stabilise.  Overall, the $E_i$ stabilise.

\textbf{(2)(b)}: As in case (1)(b), we can assume the $H^0(E_i)$ are constant and that the $f\ch_1(E_i)$ are constant.  The argument for describing $G_i$ in (2)(a) applies to $K_i$ here, allowing us to conclude $H^{-1}(K_i)=0$ and $K_i = H^0(K_i) \in \scalea{\gyoung(;;+,;;0)}$.  The first half of the long exact sequence of cohomology of \eqref{eq:prop-finiteness-eq4} now reads
\[
0 \to H^{-1}(E_i) \to H^{-1}(E_{i+1}) \to H^0(K_i) \to 0,
\]
where all the terms  are $\whPhi$-WIT$_1$ sheaves.   The Fourier-Mukai functor $\whPhi$ then takes it to a short exact sequence of sheaves
\[
0 \to \wh{H^{-1}(E_i)} \to \wh{H^{-1}(E_{i+1})} \to \wh{H^0(K_i)} \to 0\]
where $\wh{H^0(K_i)} \in \Coh^{\leq 0}(X)$.  By Lemma \ref{lem:A1over2H-1Eidescription} below, each $\wh{H^{-1}(E_i)}$ is a torsion-free sheaf.  Hence we have the inclusions
\[
  \wh{H^{-1}(E_i)} \hookrightarrow \wh{H^{-1}(E_{i+1})} \hookrightarrow (\wh{H^{-1}(E_{i+1})})^{\ast \ast}
\]
where $(\wh{H^{-1}(E_{i+1})})^{\ast \ast}$ is independent of $i$.  Thus the $H^{-1}(E_i)$ must stabilise, and so the $E_i$ themselves stabilise.



\textbf{(3)(a)}: Since $\Fc^l[1]$ is contained in $\Ac_{1,1/2}$, any object  $M \in \Ac_{1,1/2}^\circ$ must have $H^{-1}(M)=0$, i.e.\ $M = H^0(M)$.  Also, since we have the inclusion $W_{0,\whPhi} \subset \Ac_{1,1/2}$, it follows that
 \begin{equation}\label{eq:A112cCohinW1Tl}
 \Ac_{1,1/2}^\circ \cap \Coh (X) \subset W_{1,\whPhi} \cap \Tc^l.
\end{equation}
Hence  $E_i, G_i$  lie in $W_{1,\whPhi} \cap \Tc^l$ for all $i$.  Then $\ch_0 (E_i) \geq 0$ for all $i$, and we can assume $\ch_0(E_i)$ is constant while $\ch_0(G_i)=0$ for all $i$ by omitting a finite number of terms.  By \ref{para:W1xTcldecomposition}, we know each $G_i$ is a $\whPhi$-WIT$_1$ fiber sheaf.  Since $\scalea{\gyoung(;;+,;;0)} \subset \Ac_{1,1/2}$, we have $G_i \in \scalea{\gyoung(;;+,;;-)}$.  The $\Bl$-short exact sequence \eqref{eq:prop-finiteness-eq3} is then  taken by $\whPhi[1]$ to a short exact sequence in $\Coh^{\leq 1}(X)$
\[
0 \to \wh{E_{i+1}} \to \wh{E_i} \to \wh{G_i} \to 0.
\]
For any ample divisor on $X$ of the form $\omega'=\Theta + kf$ where $k$ is a positive integer,  we see that $\omega' \ch_1 (\wh{E_i})$ is a decreasing sequence of nonnegative integers, and so must become stationary, in which case the fiber sheaf $\wh{G_i}$ must be supported in dimension 0.  This implies, however, that  $\wh{G_i} \in \Coh^{\leq 0}(X) \cap \scalea{\gyoung(;+,;+)}$, forcing $G_i=0$, i.e.\ the $E_i$ eventually stabilise.

\textbf{(3)(b)}: As in (3)(a), the objects $E_i, K_i$  lie in $W_{1,\whPhi} \cap \Tc^l$ for all $i$, so \eqref{eq:prop-finiteness-eq4} is a short exact sequence of sheaves.  Since $\Coh(X)$ is noetherian, the $E_i$  eventually stabilise.
\end{proof}

\begin{lem}\label{lem:Lo14Lem6-4analogue}
Let $A \in \Ac_{1,1/2}$, and let $A_1$ denote the $\whPhi$-WIT$_1$ component of $H^0(A)$.  Then $\whPhi A_1 [1] \in \Coh^{\leq 0}(X)$.
\end{lem}

\begin{proof}
For  objects $M \in \Bl$, the property
\[
  \whPhi^1 M \in \Coh^{\leq 0}(X)
\]
is preserved under extension in $\Bl$.  Since this property is satisfied for all objects in the categories that generate $\Ac_{1,1/2}$, it is satisfied for all objects in $\Ac_{1,1/2}$.
\end{proof}


%

\begin{lem}\label{lem:A1over2H-1Eidescription}
Suppose $E \in \Ac_1^\circ = \{ E \in \Bl : \Hom (\Ac_1, E)=0\}$.  Then $H^{-1}(E)$ is locally free and $\wh{H^{-1}(E)}$ is torsion-free.
\end{lem}

\begin{proof}
 consider the exact sequence
\[
0 \to H^{-1}(E) \to H^{-1}(E)^{\ast \ast} \to Q \to 0
\]
where $Q$ is some coherent sheaf supported in dimension 0; this gives a $\Bl$-short exact sequence
\[
0 \to Q \to H^{-1}(E)[1] \to H^{-1}(E)^{\ast \ast}[1] \to 0.
\]
Since $E \in \Ac_1^\circ$, the term $Q$ must be zero, i.e.\ $H^{-1}(E)$ is locally free.

Recall from \ref{para:TlFlproperties}(v) that $H^{-1}(E)$ is $\whPhi$-WIT$_1$.  Also, that $E \in \Ac_1^\circ$ implies $\Hom (\Fc^{l,0}[1],E)=0$, and so $H^{-1}(E) \in \Fc^{l,-}$.

Suppose $\wh{H^{-1}(E)}$ has a  subsheaf $T$ that lies in $\Coh^{\leq 1}(X)$.  Let $T_i$ denote the $\Phi$-WIT$_i$ component of $T$.  The composite $T_0 \hookrightarrow T \hookrightarrow \wh{H^{-1}(E)}$ in $\Coh (X)$ is then taken by $\Phi$ to an injection of sheaves $\wh{T_0} \hookrightarrow H^{-1}(E)$.  Thus $\wh{T_0}$ is a torsion-free sheaf on $X$ and lies in $\Fc^{l,-}$ since $H^{-1}(E)$ is so.  However, since $\ch_0 (T_0)=0$, we must have $f\ch_1(\wh{T_0})=0$. This forces $\wh{T_0}$ and hence $T_0$ itself to  be zero, i.e.\ $T$ is a $\Phi$-WIT$_1$ fiber sheaf.  The inclusion $T \hookrightarrow \wh{H^{-1}(E)}$ then corresponds to an element in
\[
  \Hom (T, \wh{H^{-1}(E)}) \cong \Hom (\wh{T}[-1],H^{-1}(E)) \cong \Hom (\wh{T},H^{-1}(E)[1])
\]
where $\wh{T}=\Phi T[1]$.  Note that $\wh{T}$ is a $\whPhi$-WIT$_0$ fiber sheaf, and so is an object in $\Ac_1$.  Since $H^{-1}(E)[1]$ is a $\Bl$-subobject of $E$, which lies in $\Ac_1^\circ$, $H^{-1}(E)[1]$ itself lies in $\Ac_1^\circ$, which means the injection $T \hookrightarrow \wh{H^{-1}(E)}$ must be the zero map, i.e.\ $T=0$.  This proves that $\wh{H^{-1}(E)}$ is torsion-free.
\end{proof}

Let us now set
\begin{align*}
  \Ac_{1/2} &:= \Ac_{1,1/2} \cap \Ac_1^\circ \\
  \Ac_0 &:= \Ac_{1,1/2}^\circ,
\end{align*}
so that the torsion triple \eqref{eq:torsiontripleAlowerstar} can be rewritten as
\begin{equation}\label{eq:torsiontripleAlowerstar-v2}
(\Ac_1, \, \Ac_{1/2}, \, \Ac_0).
\end{equation}

The following is an analogue of \cite[Lemma 6.5]{Lo14}:

\begin{lem}\label{lem:Lo14Lem6-5analogue}
For $i=1,\tfrac{1}{2}, 0$ and any $F \in \Ac_i$, we have $\phi (F) \to i$.
\end{lem}

\begin{proof}
The case of $i=1$ follows from the definition of $\Ac_1$ and the computation in \ref{para:phasesofobjects}.

For $i=\tfrac{1}{2}$: take any $F \in \Ac_{1/2}$.  If $f\ch_1(F)>0$, then clearly $\phi (F) \to \tfrac{1}{2}$ and we are done.  Let us assume $f\ch_1(F)=0$ from now on.  Then $f\ch_1(H^{-1}(F))=0$, meaning $H^{-1}(F) \in \Fc^{l,0}$; however, $F \in \Ac_1^\circ$ and so $H^{-1}(F)$ must be zero, i.e.\ $F = H^0(F)$.

That $F \in \Ac_{1,1/2} \cap \Coh (X)$ with $f\ch_1(F)=0$ implies $F$ cannot have any subfactors in $\scalea{\gyoung(;;*,;+;*)}$ or $\scalea{\gyoung(;+*,;+*)}$.  Hence $F$ is a fiber sheaf where all the HN factors with respect to slope stability have $\ch_2 \geq 0$.  That  $F \in \Ac_1^\circ$ then forces $F \in \scalea{\gyoung(;;+,;;0)}$, giving us  $\phi (F) = \tfrac{1}{2}$ by  \ref{para:phasesofobjects}(2.2.2).

For $i=0$: take any $F \in \Ac_0$.  From \eqref{eq:A112cCohinW1Tl} we know $F \in W_{1,\whPhi} \cap \Tc^l$.  By \ref{para:W1xTcldecomposition}, we have a two-step filtration $F_0 \subseteq F_1 = F$ in $\Coh (X)$ where $F_0$ is a $\whPhi$-WIT$_1$ fiber sheaf while $F_1/F_0 \in \Tc^{l,0}$.  From \ref{para:phasesofobjects}-(4) we know $\phi (F_1/F_0) \to 0$, so it suffices to show $\phi (F_0) \to 0$.  Since $F \in \Ac_{1,1/2}^\circ$, we have $\Hom (\scalea{\gyoung(;;+,;;0)}, F_0)=0$, implying $F_0 \in \scalea{\gyoung(;;+,;;-)}$.  By \ref{para:phasesofobjects}(2.2.3) we have $\phi (F_0)\to 0$ as desired.
\end{proof}

\begin{lem}\label{lem:Lo14Lem6-6analogue}
An object $F \in \Bl$ is $Z^l$-semistable iff, for some $i=1, \tfrac{1}{2}, 0$, we have:
\begin{itemize}
\item $F \in \Ac_i$;
\item for any strict monomorphism $0 \neq F' \hookrightarrow F$ in $\Ac_i$, we have $\phi (F') \preceq \phi (F)$.
\end{itemize}
\end{lem}

\begin{proof}
Given Lemma \ref{lem:Lo14Lem6-5analogue}, the argument in the proof of \cite[Lemma 6.6]{Lo14}  applies.
\end{proof}

\begin{thm}\label{thm:main1}
The Harder-Narasimhan property holds for $Z^l$-stability on $\Bl$.  That is, every object $F \in \Bl$ admits a filtration in $\Bl$
\[
  F_0 \subseteq F_1 \subseteq \cdots \subseteq F_n = F
\]
where each $F_i/F_{i+1}$ is $Z^l$-semistable, and $\phi (F_i/F_{i-1}) \succ \phi (F_{i+1}/F_i)$ for each $i$.
\end{thm}

\begin{proof}
Using the torsion triple \eqref{eq:torsiontripleAlowerstar-v2},  the finiteness properties in Proposition \ref{prop:Alowerstarfinitenessproperties}, along with Lemma  \ref{lem:Lo14Lem6-6analogue}, the  argument  in the proof of \cite[Theorem 6.7]{Lo14} applies.
\end{proof}

\section{Slope stability vs Bridgeland stability on Weierstra{\ss} surfaces}\label{sec:BrivslimitBri}

The following lemma gives a relation between $Z^l$-stability and Bridgeland stability on Weierstra{\ss} elliptic surfaces:

\begin{lem}
Let $p : X \to B$ be a Weierstra{\ss} surface.  Suppose $m>0$ is such that $\Theta + kf$ is ample for all $k \geq m$, and  $\omega$ is of the form \eqref{eq:omeganotation} subject to the constraint \eqref{eq:vhyperbolaequation}.  If there exists $v_0>0$ and an object $F \in D^b(X)$ such that, for all $v >v_0$, the divisor $\omega$ is ample and $F$ lies in $\Bc_\omega$ and is $Z_\omega$-(semi)stable, then  $F$ lies in $\Bc^l$ and is $Z^l$-(semi)stable.
\end{lem}

\begin{proof}
This follows from the equivalent definitions of $\Tc^l$ and $\Fc^l$ in \ref{para:TlFlproperties}.  (See also \cite[Lemma 7.1]{Lo14})
\end{proof}

\begin{rem}
Conceivably, for any fixed Chern character, there would be `mini-walls' of Bridgeland stability as the parameter $v$ approaches infinity along the curve \eqref{eq:vhyperbolaequation}.  Results on boundedness of mini-walls similar to those in \cite{LQ} or more general results on walls on the $vu$-plane on Weierstra{\ss} surfaces will then yield equivalence between $Z^l$-stability and $Z_\omega$-stability, i.e.\ between the `limit' Bridgeland stability in this article and Bridgeland stability.  Along with Theorem \ref{thm:Lo14Thm5-analogue}, they can then produce morphisms between moduli of slope stable torsion-free sheaves and moduli of Bridgeland stable objects on Weierstra{\ss} surfaces.
\end{rem}


\bibliography{refs}{}
\bibliographystyle{plain}

                                                              \end{document}